\documentclass[a4paper,reqno, 11pt]{amsart}   	
\usepackage[DIV=12, oneside]{typearea}			
\usepackage[utf8]{inputenc}
\usepackage[T1]{fontenc}
\usepackage[english]{babel}
\usepackage[centertags]{amsmath}
\usepackage{amstext,amssymb,amsopn,amsthm}
\usepackage{mathrsfs}
\usepackage{dsfont}
\usepackage{bbm}
\usepackage{thmtools}
\usepackage{graphicx}
\usepackage[titletoc,title]{appendix}
\usepackage{etexcmds}

\usepackage[backgroundcolor=white, bordercolor=blue,
linecolor=blue]{todonotes}

\usepackage{multirow}

\usepackage{tikz}

\usepackage{pgfplots}
\pgfplotsset{compat=newest}
\usetikzlibrary{calc, patterns,angles,quotes}
\usepgfplotslibrary{fillbetween}
\definecolor{c595959}{RGB}{89,89,89}
\definecolor{c5a5a5a}{RGB}{90,90,90}

\usepackage[colorlinks=true, linkcolor=black, citecolor=black]{hyperref}
\usepackage{enumitem}
\setlist[enumerate]{itemsep=0mm}

\addto\extrasenglish{}
\addto\extrasenglish{}
\addto\extrasenglish{}

\theoremstyle{plain}
\declaretheorem[title=Theorem, parent=section]{theorem}
\declaretheorem[title=Lemma,sibling=theorem]{lemma}
\declaretheorem[title=Proposition,sibling=theorem]{proposition}
\declaretheorem[title=Corollary,sibling=theorem]{corollary}

\theoremstyle{definition}
\declaretheorem[title=Definition,sibling=theorem]{definition}
\declaretheorem[title=Remark,sibling=theorem]{remark}
\declaretheorem[title=Remark, numbered=no]{remark*}

\declaretheorem[title=Assumption, numbered=no]{assumption*}

\numberwithin{equation}{section}

\newcommand{\N}{\mathds{N}}
\newcommand{\R}{\mathds{R}}

\newcommand{\fd}{\mathfrak{d}}

\def\hmath$#1${\texorpdfstring{{\rmfamily\textit{#1}}}{#1}}

\newcommand{\cB}{\mathcal{B}}

\newcommand{\cE}{\mathcal{E}}

\newcommand{\eps}{\varepsilon}

\newcommand{\loc}{\mathrm{loc}}

\newcommand{\BIGOP}[1]
{
	\mathop{\mathchoice%
		{\raise-0.22em\hbox{\huge $#1$}}%
		{\raise-0.05em\hbox{\Large $#1$}}{\hbox{\large $#1$}}{#1}}}

\def\XXint#1#2#3{{\setbox0=\hbox{$#1{#2#3}{\int}$}
		\vcenter{\hbox{$#2#3$}}\kern-.5\wd0}}

\newcommand{\BIGboxplus}{\mathop{\mathchoice%
		{\raise-0.35em\hbox{\huge $\boxplus$}}%
		{\raise-0.15em\hbox{\Large $\boxplus$}}{\hbox{\large $\boxplus$}}{\boxplus}}}

\newcommand{\I}[1]{(\text{I}_{#1})}
\newcommand{\II} [1] {(\text{II}_{#1})}
\newcommand{\III}[1]{(\text{III}_{#1})}

\DeclareMathOperator{\dist}{dist}

\DeclareMathOperator{\diam}{diam}

\DeclareMathOperator{\supp}{supp}

\DeclareMathOperator{\tail}{Tail}
\DeclareMathOperator{\pv}{p.v.}

\renewcommand{\d}{\textnormal{d}}

\newcommand{\1}{{\mathbbm{1}}}

\usepackage{esint}
\newcommand{\norm}[1]{\left\lVert#1\right\rVert} 
\newcommand{\abs}[1]{\ensuremath{\left\vert#1\right\vert}} 
 
\newcommand{\distw}[2]{\dist(#1, #2)} 

\makeatletter
\providecommand\@dotsep{5}
\def\listtodoname{List of Todos}
\def\listoftodos{\@starttoc{tdo}\listtodoname}
\makeatother

\makeatletter
\def\namedlabel#1#2{\begingroup
	#2%
	\def\@currentlabel{#2}%
	\phantomsection\label{#1}\endgroup
}
\makeatother

\newcommand{\cdini}{C^{1,\text{\normalfont{dini}}}}

\begin{document}
	\allowdisplaybreaks
	\title{Boundary regularity and Hopf lemma for nondegenerate stable operators}
	
	\author{Florian Grube}
	
	\address{Fakult{\"a}t f{\"u}r Mathematik, Universit{\"a}t Bielefeld, Postfach 10 01 31, 33501 Bielefeld, Germany}
	\email{fgrube@math.uni-bielefeld.de}
	
	\makeatletter
	\@namedef{subjclassname@2020}{%
		\textup{2020} Mathematics Subject Classification}
	\makeatother
	
	

	\begin{abstract} 
		We prove sharp boundary Hölder regularity for solutions to equations involving stable integro-differential operators in bounded open sets satisfying the exterior $\cdini$-property. This result is new even for the fractional Laplacian. A Hopf-type boundary lemma is proven, too. An additional feature of this work is that the regularity estimate is robust as $s\to 1-$ and we recover the classical results for second order equations.
	\end{abstract}

	\maketitle
	
	\section{Introduction}
	
	The study of boundary regularity for solutions to partial differential equations involving second order operators has been a consistent topic of research in the last hundred years. 
	
	In particular, it is well known that solutions to the problem 
	\begin{equation}\label{eq:laplace}
		\begin{split}
			-\Delta u &= f \text{ in }\Omega,\\
			u&= 0 \text{ on }\partial \Omega
		\end{split}
	\end{equation}
	for a sufficiently regular domain $\Omega\subset \R^d$ are Lipschitz continuous up to the boundary. However, the regularity of the boundary $\partial \Omega$ plays a key in this result. In particular, if $\partial \Omega$ is only Lipschitz continuous, then the solution $u$ is at most Hölder continuous for some Hölder exponent, see \cite[p.71]{FeRo22}. This raises the question: what regularity of $\partial \Omega$ is required such that solutions are still globally Lipschitz? In the late 70s and 80s the optimal condition was found. In \cite{KaKh73, KaHi74} it was shown that under the assumption that $\Omega$ satisfies an exterior $\cdini$-property at every boundary point, the solution to \eqref{eq:laplace} is Lipschitz continuous on $\overline{\Omega}$. Moreover, as is proven in \cite[Theorem 2]{KaHi74}, this property is not just sufficient but also necessary. \medskip
	
	In this article, we study nondegenerate $2s$-stable integro-differential operators
	\begin{equation}\label{eq:stable_op}
		\begin{split}
			A_\mu^s u(x)&:= \pv \int\limits_{\R^d} \big(u(x)-u(x+h)\big) \nu_s(\d h)\quad \text{ for }u:\R^d\to \R, \\
			\nu_s(U)&:=(1-s)\int\limits_{\R}\int\limits_{S^{d-1}} \1_{U}(x+r\theta)\mu(\d \theta)\d r\quad \text{ for } U\in \cB(\R^d).
		\end{split}
	\end{equation}
	Here, $s\in (0,1)$, $\mu$ is a finite measure on the unit sphere $S^{d-1}$, i.e.\ there exists $\Lambda>0$ such that $\mu(S^{d-1})\le \Lambda$, and we assume that $\mu$ satisfies the nondegeneracy assumption 
	\begin{equation}\label{eq:non-degenerate}
		\lambda\le \inf\limits_{w\in S^{d-1}} \int\limits_{S^{d-1}} \abs{w\cdot \theta}^{2s}\mu (\d \theta)
	\end{equation}
	for some positive $\lambda$. These operators are translation invariant, symmetric, and scale with a factor of power $2s$, i.e.\ $A_\mu^s[u(r\cdot)](x)=r^{2s}A_\mu^su(rx)$. They are nonlocal in the sense that for a function $u:\R^d\to \R$ the support of $A_\mu^su$ may not be a subset of the support of $u$ itself. \smallskip
	
	The class of nondegenerate $2s$-stable operators includes the fractional Laplacian $(-\Delta)^s$, when $\mu$ is the uniform distribution. Moreover, the nondegeneracy condition \eqref{eq:non-degenerate} together with the finiteness of $\mu$ yields the comparability of the Fourier-symbol of $A_\mu^s$ with that of the fractional Laplacian, namely $\abs{\xi}^{2s}$. Another example in the class of operators under consideration is the sum of $1$-dimensional fractional Laplacians $\sum_{i=1}^d (-\partial_i^2)^s$. This example is recovered when choosing $\mu$ to be a sum of Dirac-measures in all coordinate directions. \smallskip
	
	We investigate the boundary regularity of solutions $u$ to 
	\begin{equation}\label{eq:main_equation_stable}
		\begin{split}
			A_\mu^s u &= f \quad \text{ in }\Omega,\\
			u&=0 \quad \text{ on }\Omega^c,
		\end{split}
	\end{equation}
	where $f\in L^\infty(\Omega)$.\smallskip
	
	In some of the following results, we need the following assumption on the L{\'e}vy-measure $\nu_s$:
	\begin{equation}\label{ass:pUB}\tag{\normalfont{pUB}}
		\begin{split}
			&\text{The measure $\nu_s$ has a density with respect to Lebesgue's measure and}\\
			&\qquad\qquad\qquad\qquad\qquad\nu_s(x)\le (1-s)\Lambda \abs{x}^{-d-2s}.
		\end{split}
	\end{equation}
	
	Now, we state our main results. 
	
	\begin{theorem}[Boundary regularity I]\label{th:boundary_regularity}
		Let $\Omega\subset \R^d$ be a bounded, open set that satisfies the uniform exterior $\cdini$-property. We fix some $s_0\in (0,1)$, $0<\lambda\le \Lambda<\infty$ and assume \eqref{ass:pUB} or $s_0>1/2$. There exists a positive constant $C$ such that for any $s\in[s_0,1)$, any measure $\mu$ on the sphere satisfying \eqref{eq:non-degenerate} and $\mu(S^{d-1})\le \Lambda$, any $f\in L^\infty(\Omega)$, and any weak solution $u$ to \eqref{eq:main_equation_stable} the solution $u$ is Hölder continuous, i.e.\ $u\in C^s(\R^d)$, and satisfies the bound
		\begin{equation}\label{eq:Cs_regularity}
			\norm{u}_{C^{s}(\R^d)}\le C \norm{f}_{L^\infty(\Omega)}.
		\end{equation}
	\end{theorem}

	The definition of $\cdini$-domains is given in \autoref{def:c1dini_domain}. The notion of the exterior $\cdini$-property is introduced in \autoref{def:ext_int_property}. 
	
	\autoref{th:boundary_regularity} enhances existing results on boundary regularity which assume a stronger $C^{1,\alpha}$-condition on the boundary of $\Omega$, see \cite{RoFe24, Ros16} and the detailed discussion in \autoref{sec:literature}. Moreover, the estimate \eqref{eq:Cs_regularity} is robust as $s\to 1-$.
	
	Whenever $s$ is smaller or equal $1/2$, we make an interesting observation in the case when the L{\'e}vy-measure $\nu_s$ is singular with respect to Lebesgue's measure, i.e.\ \eqref{ass:pUB} does not hold. In this case, we need a slightly stronger assumption on the set $\Omega$ than the exterior $\cdini$-property. Instead, we introduce the concept of $2s$-$\cdini$-domains, see \autoref{def:c1dini_2s_domain}. This class of domains is only slightly more restrictive than $\cdini$-domains and, in particular, it includes $C^{1,\alpha}$-domains for any $\alpha>0$. This distinction between singular and non-singular L{\'e}vy-measures $\nu_s$ for $s\le 1/2$ is in contrast to previous results on boundary regularity in \cite{Ros16}. 
	
	The following theorem is the resulting boundary regularity in this case. 
	
	\begin{theorem}[Boundary regularity II]\label{th:boundary_regularity_without_pUB}
		Let $s_0\in (0,1)$ and $0<\lambda\le \Lambda<\infty$. We fix a bounded open set $\Omega\subset \R^d$ that satisfies the exterior $2s_0$-$\cdini$-property uniformly. There exists a positive constant $C$ such that for any $s\in[s_0,1)$, any measure $\mu$ on the sphere satisfying \eqref{eq:non-degenerate} and $\mu(S^{d-1})\le \Lambda$, any $f\in L^\infty(\Omega)$, and any weak solution $u$ to \eqref{eq:main_equation_stable} the solution satisfies $u\in C^s(\R^d)$ and the bound \eqref{eq:Cs_regularity}. 
	\end{theorem}
	
	\begin{remark}\label{rem:solution_concept_main_results}
		Instead of weak solutions in \autoref{th:boundary_regularity} and \autoref{th:boundary_regularity_without_pUB} we can also treat continuous distributional solutions. The common ingredient is the access to a maximum principle to compare the solutions to a barrier function, see e.g.\ \cite{Sil07} and \cite[Lemma 2.3.5]{RoFe24}.
	\end{remark}

	\begin{remark}
		The different assumptions on the boundary of the domain $\Omega$ in \autoref{th:boundary_regularity} and \autoref{th:boundary_regularity_without_pUB} are due to \autoref{prop:almost_harmonic} and \autoref{lem:intermediate_estimate_1}. 
	\end{remark}

	\begin{remark}\label{rem:optimality}
		It is known that \autoref{th:boundary_regularity} is false for Lipschitz domains $\Omega\subset \R^d$. More precisely, whenever the set $\Omega$ has an inward corner the solution $u$ decays slower than $\distw{x}{\partial \Omega}^s$ in the corner. This can be observed by studying $s$-harmonic functions in cones and their decay behavior near the apex, see \cite{BaBo04} and \cite[Lemma 3.3]{Mic06}. 
	\end{remark}

	\begin{remark}
		It is an interesting problem for further research whether the assumptions on the domain $\Omega$ in \autoref{th:boundary_regularity} and \autoref{th:boundary_regularity_without_pUB}, i.e.\ the $\cdini$-property respectively $2s$-$\cdini$-property, are necessary. 
	\end{remark}
	
	One phenomenon is of particular interest to us. After choosing appropriate normalization constants, the operators $A_\mu^s$ converge to constant-coefficient, second order operators $A\cdot D^2$ in the limit $s\to 1-$. Starting with the article \cite{BBM01} and based on this observation, there has been considerable interest in Dirichlet problems like \eqref{eq:main_equation_stable} and the study of the asymptotics as $s\to 1-$. A feature of this work is that the boundary regularity estimates are robust as $s\to 1-$, see e.g.\ \eqref{eq:Cs_regularity} and \eqref{eq:distance_lower_bound_local}. \smallskip

	In this article, we also prove a Hopf-type boundary lemma for super-solutions $u$ to $A_\mu^s u\ge 0$. The classical Hopf boundary lemma, see \cite{Hop52}, says that a superharmonic, nontrivial function in a sufficiently regular domain has a positive inward normal derivative at the boundary, i.e.\ 
	\begin{equation*}
		\liminf\limits_{\eps \to 0+} \frac{u(z+\eps n_z)-u(z)}{\eps}>0
	\end{equation*}
	where $z$ is a point in the boundary of the domain and $n_z$ is the inward normal vector at $z$. We refer the interested reader to the exposition on classical Hopf boundary lemmata in \cite{ApNa22}. Hopf boundary lemmata are known for solutions to equations involving second order operators in domains satisfying the interior $\cdini$-property, see \cite{KaKh73}, \cite{KaHi74}, \cite{KaHi80}, and \cite{Lie85}. Within these references, it is also proven that the $\cdini$ assumption is sharp.\smallskip
	
	In the last decade, there have been several works in which Hopf-type boundary lemmata are proven for nonlocal operators. An early contribution in this regard is \cite[Proposition 2.7 ii)]{CRS10} which proves that a nonnegative, nontrivial $s$-harmonic function can be bounded from below by $C((x-z)\cdot n_z)^s$ at a boundary point $z$ where $u$ vanishes. The best known result for the class of nondegenerate $2s$-stable operators in terms of boundary regularity is $C^{1,\alpha}$, see \cite{RoFe24}. A more detailed exposition of nonlocal Hopf lemmata can be found in the \autoref{sec:literature}.
	
	\begin{proposition}[Hopf boundary lemma]\label{prop:hopf_with_pUB}
		Let $\Omega \subset \R^d$ be an open set, $s\in (0,1)$, and $0<\lambda\le \Lambda<\infty$. We assume one of the following three properties:
		\begin{itemize}
			\item[(1)] $\Omega$ satisfies the interior $\cdini$-property at a boundary point $z\in B_{1/2}\cap \partial \Omega$ and \eqref{ass:pUB}, 
			\item[(2)] $\Omega$ satisfies the interior $\cdini$-property at a boundary point $z\in B_{1/2}\cap \partial \Omega$ and $s>1/2$,
			\item[(3)] $\Omega$ satisfies the interior $2s$-$\cdini$-property at a boundary point $z\in B_{1/2}\cap \partial \Omega$.
		\end{itemize}
		There exist positive constants $C, \eps_0$ such that for any finite measure $\mu$ on the unit sphere satisfying \eqref{eq:non-degenerate} and any distributional super-solution $u\in C(B_1\cap \Omega)$ to 
		\begin{equation}\label{eq:frac_lap_sup_solution_localized}
			\begin{split}
				A_s u &\ge 0 \text{ in }\Omega\cap B_{1},\\
				u &\ge 0 \text{ on }\big(\Omega \cap B_{1}\big)^c,
			\end{split}
		\end{equation}
		the function $u$ is either zero or there exists a constant $C$ such that $u(x)\ge C\,\abs{x-z}^s$ for any $x\in B_{1/2}\cap \Omega$ inside of a non-tangential cone with apex at $z$. 
	\end{proposition}
	We provide a greater variety of Hopf-type boundary estimates with a more detailed analysis of the constant $C$ in \autoref{sec:hopf}.

	\subsection{Related literature}\label{sec:literature}
	
	A typical approach to prove optimal boundary regularity entails the construction of appropriate barrier functions which describe the boundary behavior of solutions to \eqref{eq:main_equation_stable}. These functions were observed in \cite{Get61}, but we also refer to \cite{Lan72}, \cite{ChSo98}, \cite{CKS10}, \cite{Dyd12}, and \cite{BGR15}. An early approach to boundary regularity, a boundary Harnack principle, was proved in \cite{Bog97}. In \cite{RoSe14}, it was proved that given a bounded Lipschitz-domain $\Omega\subset \R^d$ satisfying the exterior ball property and an essentially bounded function $f$ the solution $u$ to 
	\begin{equation*}
		\begin{split}
			(-\Delta)^s u &= f \text{ in }\Omega,\\
			u&=0 \text{ on }\Omega^c			
		\end{split}
	\end{equation*}
	is $C^s$-Hölder-continuous on $\R^d$. Moreover, this result is optimal in terms of global Hölder regularity of the solution $u$. This result was improved in \cite{Gru14, Gru15} wherein a variety of regularity results was proved for equations involving pseudo-differential operators of fractional order. The \autoref{th:boundary_regularity} for the full class of nondegenerate stable operators in $C^{1,1}$-domains was proved in \cite{Ros16}. This was extended to $C^{1,\alpha}$-domains in the recent book \cite{RoFe24}. Fine boundary regularity properties for equations involving nonlinear nonlocal integro-differential operators were treated in \cite{RoSe16}. In \cite{RoSe17}, the authors considered $2s$-stable operators where the measure $\mu$ has a bounded density with respect to the surface measure. They treated $C^{1,\alpha}$ as well as $C^1$-domains. In the latter case, they proved solutions to be Hölder continuous for some Hölder exponent. Moreover, they treated extremal operators. In the articles \cite{ChKi02, Che18, Fal22}, the authors considered the regional fractional Laplacian and proved boundary regularity for solutions to equations involving it. In the articles \cite{CaSi11b, KKP16b, AbRo20, RoFe24, RoWe24}, the authors studied operators with kernels comparable to the fractional Laplacian, i.e.\
	\begin{equation*}
		L_ku(x)= \pv \int\limits_{\R^d} (u(x)-u(y))k(x-y)\d y
	\end{equation*}
	where $\lambda \abs{h}^{-d-2s}\le k(h)\le \Lambda\abs{h}^{-d-2s}$. In particular, they proved boundary regularity of solutions to equations involving $L_k$. In \cite{AbRo20}, the authors proved optimal boundary regularity estimates of the quotient $u/\distw{x}{\partial\Omega}^s$ in $C^{k,\alpha}$-domains. In \cite{CaSi11b, KKP16b, RoFe24}, it was proved that solutions $u$ to $L_k u = f\in L^\infty(\Omega)$ and $u=0$ on $\Omega^c$ are globally Hölder-continuous with some Hölder exponent. This was extended in \cite{RoWe24}. There, the optimal $C^s$-Hölder regularity was established. The case of non-translation invariant kernels $k(x,y)$ was treated in the article \cite{KiWe24}. In $C^{1,\alpha}$-domains and under the assumption that the kernel is Hölder continuous, see \cite[(1.3)]{KiWe24}, the authors proved optimal $C^s$-regularity, see \cite[Theorem 1.6]{KiWe24}. The Dirichlet problem for the logarithmic Laplacian was studied in \cite{ChWe19}. In particular, they proved boundary regularity by constructing appropriate barriers. This was improved in the recent article \cite{HLS25} where optimal boundary regularity was shown. Boundary Hölder regularity for solutions to the Dirichlet problem involving the fractional $p$-Laplacian was studied in \cite{IMS16}. In this context of nonlinear nonlocal operators and results on boundary regularity, we also refer to \cite{KKP16b}, \cite{FSV22}, \cite{GKS23}, and \cite{IaMo24} and the references therein. Mixed local and nonlocal operators were treated in \cite{BDV22}.
	
	Higher order expansions and regularity of $u/\distw{x}{\partial \Omega}^s$ for solutions $u$ to problems like \eqref{eq:main_equation_stable} were studied in \cite{RoSe14, Gru14, Gru15,RoSe16, Ros16, RoSe17,Fal19, RoFe24, RoWe24}. \smallskip
	
	The first Hopf lemma for equations involving nonlocal operators was proved in  the article \cite[Proposition 2.7 ii)]{CRS10}. Therein, the Hopf lemma was observed for $s$-harmonic functions in smooth domains. The first inhomogeneous Hopf lemma for the fractional Laplacian was proved in \cite[Proposition 3.3]{FaJa15}, \cite[Lemma 2.7]{IMS15}. In the former, they treated equations with zeroth-order right-hand sides and domains satisfying the interior ball condition. In the latter article, the authors studied super $s$-harmonic functions in $C^{1,1}$-domains. For domains with interior ball condition and super-solutions $u$ to $(-\Delta)^s u(x)\ge c(x)u(x)$, the Hopf lemma was proved in \cite{GrSe16}. The Hopf lemma for the fractional $p$-Laplacian was studied in \cite{DeQu17, CLQ20}. The class of nonlocal operators of the form $\psi(-\Delta)$ was considered in \cite{BiLo21}. In \cite{AnCo23}, mixed local and nonlocal operators were treated. The Hopf lemma for the class of nondegenerate $2s$-stable operators in $C^{1,\alpha}$-domains was proved in the recent book \cite{RoFe24}. In \cite[Theorem 1.1]{RoWe24} and \cite[(1.8)]{KiWe24}, the authors considered operators with kernels which are comparable to $\abs{\cdot}^{-d-2s}$ in $C^{1,\alpha}$-domains. The former includes a Hopf lemma for translation invariant operators in $C^{1,\alpha}$-domains and the latter non-translation invariant operators with Hölder continuous kernels. 
		
	\subsection{The main obstruction in comparison to previous proofs} \label{sec:comparision_to_RO}
	
	A typical proof of the $C^s$-boundary regularity combines known interior regularity estimates with a decay bound of the form $\abs{u(x)}\le Cd_\Omega(x)^s=C\distw{x}{\Omega^c}^s$ for the solution $u$ to \eqref{eq:main_equation_stable}. This bound can be proven with the help of appropriate barrier functions. 
	
	Usually the barrier is constructed as $d_\Omega^s-cd_\Omega^{s+\eps}$, see \cite[Section B.2]{RoFe24}. Let us briefly explain why this choice does not work in $\cdini$-domains. 
	
	In $C^{1,\alpha}$ domains $\Omega \subset \R^d$, the main observations are 
	\begin{align}
		\abs{(-\Delta)^s d_\Omega^s }&\lesssim d_\Omega^{\alpha-s}\tag{a}\label{eq:Ls_deltas_c1alpha} \text{ in }\Omega,\\
		(-\Delta)^s d_\Omega^{s+\eps}&\lesssim -d_\Omega^{\eps-s}<0 \text{ near }\partial \Omega \text{ in }\Omega.\tag{b}\label{eq:Ls_deltasepsilon_c1alpha}
	\end{align}
	
	The property \eqref{eq:Ls_deltas_c1alpha} is mainly due to $t_+^s$ being an $s$-harmonic function in $\R_+$. The property \eqref{eq:Ls_deltasepsilon_c1alpha} essentially uses $(-\Delta)^s_{\R}t_+^{s+\eps}\lesssim -t^{\eps-s}$ in $\R_+$. Now, choosing $\eps<\alpha$ yields an appropriate subharmonic barrier. 
	
	If $\Omega$ is merely a $\cdini$-domain with a modulus of continuity $\omega$, then one proves \begin{equation*}
		\abs{(-\Delta)^sd_\Omega^s}\lesssim d_\Omega^{-s}\omega(d_\Omega)
	\end{equation*} in place of \eqref{eq:Ls_deltas_c1alpha}. But, the estimate \eqref{eq:Ls_deltasepsilon_c1alpha} remains the same. Since $\omega(t)$ may converge much slower to zero than $t^{\eps}$ as $t\to 0+$ for any $\eps>0$, a sign of $(-\Delta)^s[d_\Omega^s-cd_\Omega^{s+\eps}]$ cannot be guaranteed. Indeed, we need a replacement for $d_\Omega^{s+\eps}$, see \autoref{prop:subsolution_d}.
	
	\subsection*{Outlook}\label{sec:outlook}
	
	In \autoref{sec:prelim}, we introduce basic notation, spaces, and, most importantly, the definition of domain regularities, see \autoref{def:c1dini_domain}, \autoref{def:c1dini_2s_domain}, and \autoref{def:ext_int_property}. Robust as $s\to 1-$ interior regularity estimates are proven in \autoref{sec:interior}. Similar to the observations on the functions $t^s_+$ and $t^{s+\eps}_+$, see \autoref{sec:comparision_to_RO}, we provide a new subharmonic function in $\R_+$ which replaces $t_+^{s+\eps}$ in \autoref{sec:one_dimensional}. In \autoref{sec:higher_dimensional}, we combine the one-dimensional observations with regularized distance functions to construct the barrier functions. In \autoref{sec:boundary_regularity_dini}, we prove \autoref{th:boundary_regularity}, \autoref{th:boundary_regularity_without_pUB}, \autoref{prop:hopf_with_pUB} and, additionally, provide a variety of related boundary estimates.
	
	\subsection*{Acknowledgements} Financial support by the German Research Foundation (GRK 2235 - 282638148) is gratefully acknowledged. The author is grateful to Thorben Hensiek, Solveig Hepp, and Moritz Kassmann for valuable comments on the manuscript.
	
	\section{Preliminaries}\label{sec:prelim}
	In this section, we introduce basic definitions and notation used throughout this work. In particular, the concept of the interior-/exterior-$\cdini$-property and $2s$-$\cdini$-domains are defined. \smallskip
	
	We use $a\wedge b = \min\{a,b\}$ and $a\vee b= \max\{a,b\}$. Moreover, $a_+$ denotes the positive part of $a$, i.e.\ $a\vee 0$. We write $A\lesssim B$ if there exists a positive constant $C$ such that $A\le CB$. We use the notation $x=(x',x_d)$ with $x'\in \R^{d-1}$ and $x_d\in \R$ for $x\in \R^d$. For an open set $\Omega\subset \R^d$, we define the distance of a point $x\in \R^d$ to $\Omega^c$ via $d_\Omega(x):= \distw{x}{\Omega^c}= \inf\{\abs{x-y}\mid y\in \Omega^c \}$. In \autoref{sec:higher_dimensional}, we need a regularized version of the distance function $d_D$ for a fixed domain $D\subset \R^d$. This is denoted by $\fd$, see \eqref{eq:reg_distance_properties}. 
	
	The space of Hölder continuous functions on $\overline{\Omega}$ for a given set $\Omega$ is written as  $C^\alpha(\overline{\Omega})$. Moreover, it is equipped with the norm $\norm{u}_{C^\alpha(\Omega)}:= \norm{u}_{C^{\lfloor \alpha \rfloor}(\Omega)}+ \max_{\abs{\beta}= \lfloor\alpha\rfloor} [\partial^\beta u ]_{C^{\alpha-\lfloor \alpha\rfloor}(\Omega)}$	where
	\begin{equation*}
		[v]_{C^{s}(\Omega)}:=\sup\limits_{x,y\in \Omega} \frac{\abs{v(x)-v(y)}}{\abs{x-y}^{s}}
	\end{equation*}
	for $s\in (0,1)$. Here, we use the notation $\norm{u}_{C^k(\Omega)}:= \sum_{\abs{\beta}\le k} \sup_{x\in \overline{\Omega}}\abs{\partial^\beta u(x)}$. \smallskip
	
	The operators $A^s_{\mu}$ induce a bilinear form for a fixed set $\Omega$ given by  
	\begin{equation*}
		\cE_{\nu_s}(u,v):= \frac{1}{2}\int\limits_{\R^d}\int\limits_{\R^d} \1_{(\Omega^c\cap \Omega^c)^c}(x,x+h) (u(x)-u(x+h))(v(x)-v(x+h))\nu_s(\d h)\d x.
	\end{equation*}
	This is a direct consequence of a nonlocal Gauß-Green-type identity. The form $\cE_{\nu_s}$ allows to study operators like $A_{\mu}^s$ with Hilbert-space methods, see e.g.\ \cite{SeVa14}, \cite{FKV15}, and \cite{Rut18}. \smallskip
	
	We say that a function $u$ is a weak solution to the equation $A^s_\mu u = f$ in $\Omega$ and $u=0$ on $\Omega^c$, i.e.\ \eqref{eq:main_equation_stable}, if $u\in L^2(\Omega)$, $\cE_{\nu_s}(u,u)<\infty$, $u=0$ on $\Omega^c$, and the equation 
	\begin{equation*}
		\cE_{\nu_s}(u,v)=\int\limits_{\Omega} f(x)v(x)\d x
	\end{equation*}
	holds for any function $v\in L^2(\Omega)$ with $v=0$ on $\Omega^c$ and $\cE_{\nu_s}(v,v)<\infty$.\smallskip
	
	We say that a function $u$ is a distributional solution to the exterior value problem \eqref{eq:main_equation_stable}, if $u\in L^1(\Omega)\cap L^1(\R^d, \nu_s^\star(x)\d x)$ and 
	\begin{equation*}
		\int\limits_{\R^d}u(x)A_\mu^s\eta(x)\d x = \int\limits_{\Omega} f(x)\eta(x)\d x 
	\end{equation*}
	for any $\eta\in C_c^\infty(\Omega)$. Here, the function $\nu_s^\star$ is a weight adapted to $s,\mu$, and $\Omega$ introduced in \cite[(12)]{GrHe22a} such that $u\, A_\mu^s\eta$ is integrable over $\R^d$. In \cite{GHS24}, the measure $\nu_s^\star$ was considered. The authors proved Sobolev regularity for solutions to Dirichlet problems involving $2s$-stable operators $A_\mu^s$ with $L^2(\Omega^c, \nu_s^\star(x)\d x)$-exterior data. \smallskip
	
	Next, we introduce the properties on domains and their boundary used throughout this work. 
	
	\begin{definition}\label{def:c1dini_domain}
		We say an open set $D\subset \R^d$ is a uniform $\cdini$-domain if 
		\begin{itemize}
			\item[1)]{ there exists a modulus of continuity $\omega:[0,\infty)\to [0,\infty)$ that is continuous, nondecreasing, sublinear, i.e.\ $\limsup_{t\to \infty}\omega(t)<\infty$, satisfies $\omega(0)=0$, and 
			\begin{equation}\label{eq:dini}
				\int\limits_{0}^1 \frac{\omega(t)}{t}\d t <\infty,
			\end{equation}}
			\item[2)]{there exists a uniform localization radius $\rho>0$ and a uniform constant $C>0$ such that for any $z\in \partial D$ there exists a rotation $R$, a translation $T_z$ that maps $z$ to the origin, and a $C^1$-map $\phi:\R^{d-1}\to \R$ satisfying $\norm{\phi}_{C^1}\le C$ and
			\begin{equation*}
				\begin{split}
					T_zR D\cap B_1&= \{ (x',x_d)\in B_1(0)\mid \phi(x')>x_d \},\\
					\abs{\nabla \phi(y')-\nabla \phi(x')}&\le \omega(\abs{y'-x'}). 
				\end{split}
			\end{equation*}
				} 
		\end{itemize}
		
	\end{definition}

	\begin{definition}\label{def:c1dini_2s_domain}
		We say an open set $D\subset \R^d$ is a uniform $2s$-$\cdini$-domain if $D$ satisfies all assumptions from \autoref{def:c1dini_domain} but \eqref{eq:dini} is replaced with 
		\begin{equation}\label{eq:dini_2s}
			\int\limits_{0}^1 \frac{\omega(t)}{t} \frac{\omega(t)^{2s-1}-\omega(1)^{2s-1}}{1-2s}\d t <\infty.
		\end{equation}
		If $s=1/2$, then this is to be understood as 
		\begin{equation*}
			\int\limits_{0}^1 \frac{\omega(t)}{t} \ln\big(\frac{\omega(1)}{\omega(t)}\big)\d t <\infty.
		\end{equation*}
	\end{definition}
	\begin{remark}\label{rem:special_dini_s_ge_12}
		If $s>1/2$, then the assumption \eqref{eq:dini_2s} is equivalent to \eqref{eq:dini}. For all $0<s<1$, the assumption \eqref{eq:dini_2s} implies \eqref{eq:dini}. 
	\end{remark}
	\begin{lemma}\label{lem:special_2s_0_dini_bigger_s}
		If the function $\omega$ satisfies \eqref{eq:dini_2s} for some $s_0$, then it also satisfies $\eqref{eq:dini_2s}$ for any $s\ge s_0$.
	\end{lemma}
	\begin{proof}
		For $s_0\ge 1/2$ the statement is trivial. It remains to prove the statement for $0<s_0<s\le 1/2$. In the case $s<1/2$, notice that 
		\begin{align*}
			\lim\limits_{t\to 0+}\frac{\omega(t)^{2s_0-1}-\omega(1)^{2s_0-1}}{1-2s_0} \frac{1-2s}{\omega(t)^{2s-1}-\omega(1)^{2s-1}}= \lim\limits_{t\to 0+} \frac{1-2s}{1-2s_0}\frac{\omega(t)^{2s_0-1}}{\omega(t)^{2s-1}}= \infty
		\end{align*}
		since $\omega(0)=0$. Thus, there exists $t_1>0$ such that 
		\begin{equation*}
			\frac{\omega(t)^{2s_0-1}-\omega(1)^{2s_0-1}}{1-2s_0} \ge \frac{\omega(t)^{2s-1}-\omega(1)^{2s-1}}{1-2s}
		\end{equation*}
		for all $0<t<t_1$. For $s=1/2$ we argue similarly.
	\end{proof}
	
	\begin{definition}\label{def:ext_int_property}
		We say an open set $\Omega\subset \R^d$ satisfies the
		\begin{itemize}
			\item[i)] exterior $\cdini$-property at a boundary point $z\in \partial \Omega$ if there exists a uniform $\cdini$-domain $D\supset \Omega$ that touches $\Omega$ only in $z$,
			\item[ii)] interior $\cdini$-property at a boundary point $z\in \partial \Omega$ if there exists a uniform $\cdini$-domain $D\subset \Omega$ that touches $\Omega$ only in $z$.
			\item[iii)] The respective interior / exterior $2s$-$\cdini$-property are equivalently defined using \autoref{def:c1dini_2s_domain}.
			\item[iv)] An open set satisfies the interior / exterior $\cdini$/$2s$-$\cdini$-property uniformly if there exists a uniform modulus of continuity, localization radius, and diameter of the $\cdini$/$2s$-$\cdini$-domains for the interior / exterior $\cdini$-property at all boundary points. 
		\end{itemize}
	\end{definition}
	The definitions $i)$ and $ii)$ from \autoref{def:ext_int_property} are just as in \cite[p.346]{Lie85}.

	\section{Robust interior regularity}\label{sec:interior}
	
	In this section, we prove interior regularity results which, together with the boundary estimates in \autoref{sec:boundary_regularity_dini}, prove the main boundary regularity estimate, see \eqref{eq:Cs_regularity}.\smallskip

	The proofs in this section are robust in $s$ adaptations of \cite[Lemma 3.1, Proposition 3.3, Theorem 1.1 (a)]{Ros16}. Robust estimates like these for a different class of nonlinear problems were proved in \cite[Lemma 7.5]{RoSe16}.\smallskip
	
	The proof of the interior regularity is based on a compactness result and a Liouville-type theorem. For this, we need the following lemma to extract a weakly convergent subsequence from a blow up sequence constructed in \autoref{lem:interior_main_ingredient}.
	
	\begin{lemma}[Subsequence extraction]\label{lem:blow_up}
		Let $s_0\in (0,1)$, $0<\lambda\le \Lambda<\infty$ be fixed. Let $\{s_k\}$ be a sequence of real numbers in $[s_0,1)$ and $\{\mu_k\}$ be a sequence of finite measures on the unit sphere $S^{d-1}$ satisfying $\mu(S^{d-1})\le \Lambda$ and \eqref{eq:non-degenerate}. There exists a subsequence $\{k_n\}$ such that the corresponding stable operators $\{A_{\mu_k}^{s_k}\}$ converge distributionally either to an operator $A_{\mu}^s$ of type \eqref{eq:stable_op}, where $\mu$ satisfies $\mu(S^{d-1})\le \Lambda$ and the nondegeneracy assumption \eqref{eq:non-degenerate}, or to an elliptic operator of second order with constant coefficient $-A\cdot D^2 $ and ellipticity constants $\lambda,\Lambda$.\smallskip
		
		Moreover, let $\Omega\subset \R^d$ be a bounded domain, $\{u_k:\R^d\to \R\}, \{f_k:\Omega\to \R\}$ be sequences of functions satisfying $A_{\mu_k}^{s_k}u_k=f_k$ in the distributional sense in $\Omega$. Assume further that these functions satisfy 
		\begin{itemize}
			\item[i)] $u_k\to u$ locally uniformly in $\R^d$, 
			\item[ii)] $f_k\to f$ uniformly in $\Omega$,
			\item[iii)] there exist $C$, $\eps$ independent of $k$ such that $u_k$ satisfies the uniform bound \begin{equation*}
				\abs{u_k(x)}\le C(1+\abs{x})^{2s_k-\eps}.
			\end{equation*}
		\end{itemize}
		Then, the limit $u$ satisfies either $A_{\mu}^su = f$ or $A\cdot D^2 u = f$ in $\Omega$ in the distributional sense. 
	\end{lemma}
	The proof is a minor adaption of \cite[Lemma 3.1]{Ros16}. Here, it is essential to know that operators like $A_{\mu}^s$ converge to second order operators as $s\to 1-$. This is discussed e.g.\ in \cite{BBM01}, \cite{FKV20}. 
	\begin{proof}
		Every time we switch to a subsequence, we simply keep referring to it as the original sequence to ease notation. \smallskip
		
		Since $\{s_k\}$ is bounded, we find a convergent subsequence with limit $s\in [s_0,1]$. We will need to distinguish the cases $s<1$ and $s=1$. Since the space of probability measures on the sphere is weakly compact, we find a subsequence of $\{\mu_k\}$ which converges weakly to a measure $\mu$ satisfying $\mu(S^{d-1})\le \Lambda$ and \eqref{eq:non-degenerate}. We define the matrix $A:= (a_{ij})$ via 
		\begin{align*}
			a_{ij}:= \frac{1}{2}\int\limits_{S^{d-1}} \theta_i\theta_j \mu(\d \theta).
		\end{align*}
		If $s<1$, then for any test function $\eta\in C_c^\infty$ the convergence $A_{\mu_k}^{s_k}\eta(x)\to A_{\mu}^s \eta(x)$ holds due to the weak convergence of the measures $\{\mu_k\}$. If $s=1$, then we decompose $	A_{\mu_k}^{s_k}\eta(x)$ as $\I{}+\II{}$, where
		\begin{align*}
			\I{}:= (1-s_k)\int\limits_{(-1,1)} \int\limits_{S^{d-1}}\frac{\eta(x)-\eta(x+r\theta)}{\abs{r}^{1+2s_k}}\mu_k(\d \theta)\d r.
		\end{align*}
		Thus, the term $\II{}$ vanishes in the limit $k\to \infty$ since
		\begin{align*}
			\abs{\II{}}\le (1-s_k) 2 \norm{\eta}_{L^\infty}\mu(S^{d-1}) 2 \frac{1}{2s_k}.
		\end{align*}
		For the term $\I{}$, we use a Taylor expansion of the test function $\eta$. Due to symmetry, the first order term vanishes. We find 
		\begin{align*}
			\I{}= -(1-s_k)\frac{1}{2}\int\limits_{(-1,1)} \int\limits_{S^{d-1}}\frac{r\theta \cdot D^2\eta(x)\cdot r\theta+R(r)}{\abs{r}^{1+2s_k}}\mu_k(\d \theta)\d r.
		\end{align*}
		The term with the remainder $R(r)$ vanishes with the same arguments as for $\II{}$. Thus, the term $\I{}$ converges to 
		\begin{align*}
			-\frac{1}{2}\int\limits_{S^{d-1}}\theta \cdot D^2\eta(x)\cdot \theta\mu(\d \theta)
		\end{align*}
		as $k\to \infty$ due to the weak convergence of $\{\mu_k\}$ which is the desired result. \smallskip
		
		From now on we simply denote the limit operator of $A_{\mu_k}^{s_k}$ by $L$ instead of switching between $A_{\mu}^s$ and $-A\cdot D^2$. Due to the regularity of $\eta$, the above arguments, and using dominated convergence, the sequence $A_{\mu_k}^{s_k}\eta$ converges locally uniformly to $L$. Finally, combining \cite[(26), (12)]{GrHe22a} with the growth bound iii) and dominated convergence leads to 
		\begin{align*}
			\int\limits_{\R^d} u_k(x) A_{\mu_k}^{s_k}(x)\d x\to \int\limits_{\R^d} u L \eta(x)\d x.
		\end{align*}
		Due to the uniform convergence of $\{f_k\}$ and $A_{\mu_k}^{s_k}u_k=f_k$ distributionally in $\Omega$, the result follows.
	\end{proof}

	The following proposition is an adaptation of \cite[Proposition 3.3]{Ros16}. Here, the dependencies of the constant in the estimate \eqref{eq:interior_main_ingredient} are crucial. It is the main ingredient to prove \autoref{th:interior_s_regularity}.
	
	\begin{lemma}\label{lem:interior_main_ingredient}
		Let $s_0\in (0,1)$, $0<\lambda\le\Lambda<\infty$ be fixed. There exists a positive constant $C= C(s_0,\lambda,\Lambda)$ such that for all $s\in[s_0,1)$, operators $A_{\mu}^s$ as in \eqref{eq:stable_op} satisfying \eqref{eq:non-degenerate}, $f\in L^\infty(B_1)$, $u\in C_c^\infty(\R^d)$ solving $A_{\mu}^s u = f $ in $B_1$ in the pointwise sense the bound
		\begin{equation}\label{eq:interior_main_ingredient}
			[u]_{C^{\alpha_i(s)}(B_{1/2})}\le C\big( \norm{f}_{L^\infty(B_1)}+ \norm{u}_{C^{\beta_i(s,s_0)}(\R^d)} \big)
		\end{equation}
		holds for $i=1$ or $i=2$. Here, $\alpha_1(3/4)= 3/4$, $\alpha_2(3/4)=5/4$, $\beta_1(s_0,3/4):=s_0/2$, $\beta_2(s_0,3/4):=9/8$, and for both $i\in\{1,2\}$ and $s\in (s_0,1)\setminus\{3/4\}$
		\begin{align*}
			\alpha_i=\alpha_i(s):= \begin{cases}
				s &\text{ if } s\in[s_0,3/4)\\
				5/3\,s &\text{ if }s\in (3/4,1)
			\end{cases},\quad	\beta_i=\beta_i(s)=\beta_i(s,s_0):= \begin{cases}
				s_0/2 & \text{ if } s\in [s_0,3/4)\\
				9/8 & \text{ if }s\in (3/4,1)
			\end{cases}.
		\end{align*}
	\end{lemma}
	
	\begin{proof}
		We prove the result by contradiction, very close to \cite[Proposition 3.3]{Ros16}, and adapt whenever necessary. Assume that the statement is false. Thus, there exists $s_0\in (0,1)$, $0<\lambda\le \Lambda<\infty$ and a sequence of real numbers $\{s_k\}\subset [s_0,1)$, a sequence of measures $\{\mu_{k}\}$ on the unit sphere satisfying \eqref{eq:non-degenerate} and $\mu(S^{d-1})\le \Lambda$, a sequence of functions $f_k\in L^\infty(B_1)$, $u\in C_c^\infty(\R^d)$ solving $A_{\mu_k}^{s_k}u_k=f_k$ in $B_1$ and 
		\begin{align}
			\norm{f_k}_{L^\infty(B_1)}+ \norm{u_k}_{C^{\beta_i(s_k)}(\R^d)}&\le 1, \label{eq:interior_main_ingredient_contra1}\\
			\norm{u_k}_{C^{\alpha_i(s_k)}(B_{1/2})}&\ge k.\label{eq:interior_main_ingredient_contra2}
		\end{align}
		for both $i=1$ and $i=2$.
	
		\textit{Claim a.} The expression 
		\begin{equation*}
			\sup\limits_{k\in \N} \sup\limits_{z\in B_{1/2}} \sup\limits_{r>0} r^{-\alpha_i(s_k)+\beta_i(s_k)} [u_k]_{C^{\beta_i(s_k)}(B_r(z))}
		\end{equation*}
		is infinite for both $i=1$ and $i=2$.\smallskip
		
		Assume the contrary, then there exists a positive constant $c_1$ such that for any $k, z, r, \abs{h}<r$ we know 
		\begin{equation*}
			r^{-\alpha_i(s_k)+\beta_i(s_k)}\frac{\abs{u_k(z+h)-u_k(z)}}{\abs{h}^{\beta_i(s_k)}}\le c_1.
		\end{equation*}
		Picking $\abs{h}=r/2$ is a contradiction to \eqref{eq:interior_main_ingredient_contra2}. This proves claim a.\smallskip
		
		Now, we choose $i$. Since the sequence $\{s_k\}$ is bounded, we find a convergent subsequence with limit $s\in[s_0,1]$. If $s\ge 3/4$, then we may pick another subsequence such that $s_k\ge 3/4$ for all $k$. If $s<3/4$, then we pick another subsequence such that all $s_k$ are strictly smaller than $3/4$. In both cases, we stick to the notation of the original sequence when referring to the aforementioned subsequence. Now, if $s\ge 3/4$, then we fix $i=2$. Else we choose $i=1$. \medskip
		
		We define the function $\kappa(r):[0,\infty)\to [0,\infty)$ via
		\begin{align*}
			\kappa(r):= \sup\limits_{k\in \N} \sup\limits_{z\in B_{1/2}} \sup\limits_{\tilde{r}>r} \tilde{r}^{-\alpha_i(s_k)+\beta_i(s_k)}  [u_k]_{C^{\beta_i(s_k)}(B_{\tilde{r}}(z))}.
		\end{align*}
		Notice that $\kappa$ is nonincreasing and $\kappa(r)$ is finite for any $r>0$ due to \eqref{eq:interior_main_ingredient_contra1}.\smallskip
		
		By the definition of $\kappa$, for any $n\in \N$ the term $\kappa(1/n)$ yields $\tilde{r}_n\ge 1/n$, $k_n\in \N$, and $z_n\in B_{1/2}$ such that 
		\begin{equation}\label{eq:interior_main_ingredient_new_sequence}
			\tilde{r}_n^{-\alpha_i(s_{k_n})+\beta_i(s_{k_n})}[u_{k_n}]_{C^{\beta_i(s_{k_n})}(B_{\tilde{r}_n}(z_n))}\ge \frac{1}{2}\kappa(1/n)\ge \frac{\kappa(\tilde{r}_n)}{2}.
		\end{equation}
		Now, we define a sequence of polynomials $\rho_n$ of order $i$ that are  minimizing 
		\begin{equation*}
			\int\limits_{B_{\tilde{r}_n}(z_n)} \abs{u_{k_n}(x)-\rho_n(x-z_n)}^2\d x = \min\limits_{\substack{q \text{ polynomial}\\	\text{of order } i-1}}\int\limits_{B_{\tilde{r}_n}(z_n)} \abs{u_{k_n}(x)-q(x-z_n)}^2\d x,
		\end{equation*} 
		i.e.\ the projection of $L^2$ to the subspace of $(i-1)$-th order polynomials. If $s< 3/4$, then $\rho_n= \fint_{B_{r_n}(z_n)} u_{k_n}(x)\d x$. If $s\ge 3/4$, then \begin{equation*}
			\rho_n(x-z_n)= \fint_{B_{r_n}(z_n)} u_{k_n}(y)\d y+\sum_{j=1}^{d}\frac{d+2}{r_n^2}\fint_{B_{r_n}(z_n)} u_{k_n}(y)(y_j-z_j)\d y \, (x_j-z_j).
		\end{equation*}
		
		Furthermore, we define the sequence of functions $\{v_n\}$ via
		\begin{equation*}
			v_n(x):= \frac{u_{k_n}(z_n-\tilde{r}_n x) -\rho_n(x-z_n) }{\tilde{r}_n^{\alpha_i(s_{k_n})} \kappa(\tilde{r}_n)}. 
		\end{equation*}
		This sequence $\{v_n\}$ will be the blow up sequence to which we apply \autoref{lem:blow_up}. By definition of the polynomial $\rho_n$, the functions $v_n$ have zero mean over $B_{1}(0)$. Additionally, $v_n$ is orthogonal to all polynomials of order $1$. Moreover, by definition of $v_n$ and $\kappa(\tilde{r}_n)$ we find 
		\begin{equation}\label{eq:interior_main_ingredient_v_n_lower_hoelder}
			[v_n]_{C^{\beta_i(s_{k_n})}(B_1)}\ge \frac{\tilde{r}_n^{\beta_i(s_{k_n})} [u_{k_n}]_{C^{\beta_i(s_{k_n})}(B_{\tilde{r}_n}(z_n))} }{\tilde{r}_n^{\alpha_i(s_{k_n})} \kappa(\tilde{r}_n)}\ge \frac{1}{2}.
		\end{equation}
		Next, we derive an upper bound on the Hölder norm of $v_n$. Using the definition of $\kappa$ we derive for any $R\ge 1$
		\begin{equation}
			\begin{split}\label{eq:interior_main_ingredient_v_n_upper_hoelder}
				[v_n]_{C^{\beta_i(s_{k_n})}(B_R)}&= R^{\alpha_i(s_{k_n})-\beta_i(s_{k_n})} \frac{[u_{k_n}]_{C^{\beta_i(s_{k_n})}(B_{Rr_n}(z_n)) }}{\kappa(r_n) (Rr_n)^{\alpha_i(s_{k_n})-\beta_i(s_{k_n})}}\\
				&\le \frac{\kappa(Rr_n)}{\kappa(r_n)} R^{\alpha_i(s_{k_n})-\beta_i(s_{k_n})}\le R^{\alpha_i(s_{k_n})-\beta_i(s_{k_n})}.
			\end{split}
		\end{equation}
		Here, we used that $\kappa$ is nondecreasing. \smallskip
		
		\textit{Claim b.1.} If $i=1$, then the functions $\{v_n\}$ are uniformly bounded on $B_1$, i.e.\ 
		\begin{equation*}
			\norm{v_n}_{C(B_1)}\le 2^{s_0/2}.
		\end{equation*} 
		
		Since the function $v_n$ is continuous and has zero mean on $B_1$ there exists $y_0\in B_1(0)$ such that $v_n(y_0)=0$. Due to the previous upper bound on the $C^{s_0/2}$-seminorm of $v_n$, see \eqref{eq:interior_main_ingredient_v_n_upper_hoelder}, we find for any $x\in B_1$
		\begin{align*}
			\abs{v_n(x)}\le 2^{s_0/2} \frac{\abs{v_n(x)-v_n(y_0)}}{\abs{x-y_0}^{s_0/2}}\le 2^{s_0/2}.
		\end{align*} 
		This proves claim b.1.\medskip
		
		\textit{Claim b.2.} If $i=2$, then the functions $\{v_n\}$ are uniformly bounded in $C^{1}(\overline{B_1})$. \smallskip
		
		Since $i=2$, the constant $\beta_i$ equals $9/8$. Now, we prove a uniform bound on the gradient of $v_n$. Let $j$ be a natural number between $1$ and $d$. Since $\partial_j v_n$ has zero mean over $B_1$ and is continuous, there exists $y_0\in B_1$ such that $\partial_j u(y_0)=0$. Just as in the proof of claim b.1 we deduce using the uniform bound of the $C^{9/8}$-seminorm of $v_n$, see \eqref{eq:interior_main_ingredient_v_n_upper_hoelder}, that $\abs{\partial_j v_n(x)}\le 2^{9/8}$ for any $x\in B_1$. Since $j$ was arbitrary, this yields a uniform bound on $[v_n]_{C^1(B_1)}$. Now, we repeat the arguments in the proof of claim b.1 but replace the seminorm $C^{s_0/2}$ with $C^{0,1}$ to deduce the claim b.2.\medskip 
	
		\textit{Claim c.} The sequence $\{v_n\}$ converges in $C^{\beta_i(s)/2}_{\loc}(\R^d)$ to a function $v\in C_{\loc}^{\beta_i(s)}(\R^d)$. The function $v$ satisfies 
		\begin{align}
			[v]_{C^{\beta_i}(B_R)}\le R^{\alpha_i(s)-\beta_i} \text{ for all }x\in \R^d,\\
			Lv = 0 \text{ in }\R^d \text{ in the distributional sense.}
		\end{align}
		Here, $L$ is either an operator of the type \eqref{eq:stable_op} or an elliptic second order operator with constant coefficients. \smallskip
		
		The convergence of a subsequence of $\{v_n\}$ to some function $v\in C_{\loc}^{\beta_i(s)}(\R^d)$ follows from the Arzel{\`a}-Ascoli theorem combined with claim b.1 respectively claim b.2 . Note that we already picked a subsequence such that the sequence $\{s_{k_n}\mid n\in \N\}$ converges to $s\in [s_0,1]$. Next, the equation \eqref{eq:interior_main_ingredient_v_n_upper_hoelder} leads to 
		\begin{equation}\label{eq:interior_main_ingredient_hoelder_growth_v}
			[v]_{C^{\beta_i(s)}(B_R)}\le R^{\alpha_i(s)-\beta_i(s)}
		\end{equation}
		for any $R\ge 1$. Furthermore, since $v$ has zero mean on $B_1$, there exists $y_0\in B_1$ such that $v(y_0)=0$. Thus, the bound
		\begin{equation}\label{eq:growth_bound_v_s}
			\abs{v(x)}= \abs{v(x)-v(y_0)}\le (1+\abs{x-y_0})^{\alpha_i(s)-\beta_i(s)} \abs{x-y_0}^{\beta_i(s)}\le (2+\abs{x})^{\alpha_i(s)}
		\end{equation}		
		holds. Since $\norm{f}_{L^\infty(B_1)}\le 1$, we estimate using the definition of $v_n$
		\begin{align*}
			\abs{Lv_n(x)}= \frac{\norm{f}_{L^\infty(B_1)}}{\kappa(\tilde{r}_n)} \le \frac{1}{\kappa(\tilde{r}_n)} \text{ for any } \abs{x}\le 1/\tilde{r}_n.
		\end{align*}
		Thus, the sequence $\{Lv_n\}$ converges locally uniformly to $0$ in the limit $n\to \infty$. \smallskip
		
		Using \autoref{lem:blow_up}, the sequence of operators $A_{\mu_{k_n}}^{s_{k_n}}$ converges weakly to $L$ which is either $A_\mu^s$ or $A\cdot D^2$. Moreover, the function $v$ satisfies $Lv=0$ in $\R^d$ distributionally due to the same lemma. This proves the claim c.\medskip
		
		By the translation invariance, for any $h\in \R^d$ the function $w_h(x):= v(x+h)-v(x)$ also satisfies $Lw_h=0$ in the distributional sense. Furthermore, it satisfies the growth bound
		\begin{equation}\label{eq:interior_main_ingredient_sublinear_growth}
			\abs{w_h(x)}\le \abs{h}^{\beta_i(s)}[v]_{C^{\beta_i(s)}(B_{1+\abs{h}+\abs{x}})}\le \abs{h}^{\beta_i(s)}\big(1+\abs{h}+\abs{x}\big)^{\alpha_i(s)-\beta_i(s)}.
		\end{equation}
		Note that this bound is sublinear in the variable $x$ since $\alpha_i(s)-\beta_i(s)<1$.
		Now, we are in the position to apply Liouville's theorem. If $s=1$, i.e.\ $L= A\cdot D^2$ where $A$ is elliptic, then the classical Liouville theorem with the strict sublinear bound \eqref{eq:interior_main_ingredient_sublinear_growth} proves that $w_h$ is constant. If $s<1$, i.e.\ $L= A_{\mu}^s$, then using the Liouville-type theorem \cite[Theorem 2.1]{Ros16} also leads to $w_h$ being constant. \smallskip
		
		Let $\tilde{c}_h$ be the constants obtained from the previous application of Liouville's theorem to the function $w_h$. It is easy to see that \begin{equation*}
			v(x+h)-v(x)=w_h(x)=\tilde{c}_h=w_h(0)=v(h)-v(0).
		\end{equation*} 
		Due to this and the growth bound \eqref{eq:growth_bound_v_s}, $v$ is a polynomial of order at most $1$. More precisely, if $i=1$, then $v$ is a constant. Since $v$ has mean zero over $B_1$, we find that $v$ is zero if $i=1$. If $i=2$, then it is a polynomial of order $1$. Since its partial derivatives have zero mean over $B_1$, it is a constant. Then arguing as in the case $i=1$ yields $v=0$ in both cases. 
		
		Finally, $v=0$ is a contradiction to \eqref{eq:interior_main_ingredient_v_n_lower_hoelder} in the limit $n\to \infty$. 
	\end{proof}

	\begin{theorem}\label{th:interior_s_regularity}
		Let $s_0 \in (0,1)$, $0<\lambda\le \Lambda<\infty$ be fixed. There exists a positive constant $C=C(s_0,\lambda,\Lambda,d)$ such that for all $s\in [s_0,1)$, operators $A_{\mu}^s$ as in \eqref{eq:stable_op} satisfying \eqref{eq:non-degenerate}, $f\in L^\infty(B_{3/2})$, $u\in L^\infty(\R^d)$ a distributional solution to $A_\mu^s u = f$ in $B_{3/2}$ the function $u$ satisfies 
		\begin{equation}\label{eq:interior_s_regularity}
			\norm{u}_{C^s(B_{1/2})}\le C\big(\norm{f}_{L^\infty(B_1)}+ \norm{u}_{L^\infty(B_2)} + \sup\limits_{y\in B_1} \tail_{\nu_s}(u;y) \big)
		\end{equation}
		where the tail is defined by 
		\begin{equation}\label{eq:tail}
			\tail_{\nu_s}(u;y):= \int\limits_{B_{1/2}(0)^c} \abs{u(y+r\theta)}\nu_s(\d \theta). 
		\end{equation}
	\end{theorem}
	
	\begin{proof}
		We extend $f$ to $\R^d$ by setting $f=0$ on $B_{3/2}^c$. Let $\psi\in C_c^\infty(\R^d)$ be a nonnegative smooth bump function with $\psi=1$ in $B_{3/2}$ and $\psi=0$ on $B_2^c$, $[\psi]_{C^{s_0/2}}\le c_1$. Further, let $\eta\in C_c^\infty(\R^d)$ be nonnegative with $\supp\eta\subset B_1(0)$ and $\norm{\eta}_{L^1}=1$, and $\eta_\eps:= \eps^{-d}\eta((\cdot)/\eps)$ an approximate identity. We consider the functions $v_\eps:= (u\ast \eta_\eps)\,\psi$. Due to the translation invariance of the operators $A_{\mu}^s$ this function satisfies for any $x\in B_1$ 
		\begin{align*}
			A_{\mu}^s v_\eps(x)= f\ast \eta_\eps(x)  + \int\limits_{B_{1/2}(0)^c} (1-\psi(x+h))u\ast \eta_\eps(x+h)\nu_s(\d h)=: f_\eps(x).
		\end{align*}
		Using Young's inequality, the new inhomogeneity $f_\eps$ may be bounded as
		\begin{align*}
			\norm{f_\eps}_{L^\infty(B_1)}\le \norm{f}_{L^\infty(B_{1+\eps})}+ \sup\limits_{y\in B_1} \tail_{\nu_s}(u\ast \eta_\eps;y).
		\end{align*}
		
		An application of \autoref{lem:interior_main_ingredient} yields a positive constant $C=C(s_0,\lambda,\Lambda, d)$ such that for $i=1$ or $i=2$
		
		\begin{align}\label{eq:interior_proof_final_step}
			\begin{split}
				\norm{u\ast \eta_\eps}_{C^{\alpha_i(s)}(B_{1/2})}&=\norm{v_\eps}_{C^{\alpha_i(s)}(B_{1/2})}\le C\big( \norm{f_\eps}_{L^\infty}(B_1)+\norm{v_\eps}_{C^{\beta_i(s,s_0)}(\R^d)} \big)\\
				&\le C\big(\norm{f}_{L^\infty(B_{1+\eps})}+ \sup\limits_{y\in B_1}\tail_{\nu_s}(u\ast \eta_\eps;y) \\
				&\quad+c_1 \norm{u\ast \eta_\eps}_{L^\infty(B_2)}+ \norm{u\ast \eta_\eps}_{C^{\beta_i(s,s_0)}(B_2)} \big).
			\end{split}
		\end{align}
		Since $\beta_i(s,s_0)<\beta_i(s,s_0)+\min\{ 1/8,s_0/2 \}\le \alpha_i(s)$, scaling arguments just as in the proof of \cite[Theorem 1.1 (a)]{Ros16} allow us to absorb the term $\norm{u\ast \eta_\eps}_{C^{\beta_i(s,s_0)}(B_2)}$ in the inequality \eqref{eq:interior_proof_final_step} to the left-hand side. Finally, we take the limit $\eps \to 0+$ which yields $u\in C^{\alpha_i(s)}(\overline{B_{1/2}})$ and the desired regularity estimate since $\alpha_i(s)\ge s$. 	 
	\end{proof}

	\begin{remark}
		Higher order interior regularity estimates, i.e.\ bounds on $\norm{u}_{C^{2s+\alpha}}$, as in \cite[Proposition 3.2, Proposition 3.3]{Ros16} can be proved robust as $s\to 1-$ with the same adaptations as in \autoref{lem:blow_up}, \autoref{lem:interior_main_ingredient}, and \autoref{th:interior_s_regularity}. But note that the respective constants blow up as $\alpha+2s$ approaches an integer value. For this reason, we reduced the \autoref{th:interior_s_regularity} to interior $C^s$-regularity. 
	\end{remark}

	\section{A subharmonic function in one dimension}\label{sec:one_dimensional}
	In this section, we construct a sub-solution to the fractional Laplacian in one dimension. In contrast to \cite{RoFe24}, we need a function that is much closer to the well-known $s$-harmonic function $(\cdot)_+^s$ in the half line than $(\cdot)_+^{s+\eps}$ as observed in \autoref{sec:comparision_to_RO}. Rather, we will pick a small perturbation of $(\cdot)^s_+$, say $(\cdot)_+^s\zeta(\cdot)$. We will collect all the assumptions needed for $\zeta$. \smallskip
	 
	The function $\zeta\in C([0,\infty))\cap C^2((0,\infty))\cap C^{3}((0,2))$ needs to satisfy:
	\begin{itemize}
		\item[\namedlabel{ass:zeta_pic}{\normalfont{(\textbf{Z0})}}]{The function $\zeta$ is positive in $(0,\infty)$, increasing, concave, and $\zeta(0)=0$.}
		 \item[\namedlabel{ass:zeta_3rd}{\normalfont{(\textbf{Z1})}}]{There exists constants $c_1>0$, $t_0\in(0,1)$ such that \begin{equation}\label{eq:assumption_zeta_d_1}
		 		t^2\zeta^{(3)}(t)\ge -c_1 \zeta^\prime(t) \text{ for any $0<t<t_0$ }.
		 \end{equation} } 
		 \item[\namedlabel{ass:zeta_ndecr}{\normalfont{(\textbf{Z2})}}]{The map $t\mapsto t\zeta^\prime(t)$ is nondecreasing.}
		 \item[\namedlabel{ass:zeta_iota}{\normalfont{(\textbf{Z3})}}]{There exist positive constants $\iota$, $t_0$ such that the map $t\mapsto \frac{\zeta(t)}{t^{\iota}}$ is monotonically decreasing in $(0,t_0]$.}
		 \item[\namedlabel{ass:zeta_growth}{\normalfont{(\textbf{Z4})}}]{There exists a positive constant $c_2$ such that $\zeta(t)\le c_2(1+t)^{s/2}$.}
	\end{itemize}
	
	The next proposition proves that $(-\Delta)_{\R}^s[t_+^s\,\zeta(t)]$ has a sign near zero and a possibly stronger singularity than $t^{\eps-s}$ at $t=0$. 
	\begin{proposition}\label{prop:subsolution_d_1}
		Let $\zeta\in C([0,\infty))\cap C^2((0,\infty))\cap C^{3}((0,2))$ such that \ref{ass:zeta_pic}, \ref{ass:zeta_3rd}, \ref{ass:zeta_ndecr}, and \ref{ass:zeta_growth} hold. There exists a constant $C>0$ such that for all $s\in (0,1)$ and all $0<t<t_0$
		\begin{equation*}
			(-\Delta)_{\R}^s[(\cdot)_+^s\,\zeta](t)\le -s^{6}Ct^{1-s} \zeta^\prime(t).
		\end{equation*}
		Here, the operator $(-\Delta)_{\R}^s$ is the fractional Laplacian in one-dimension given by 
		\begin{equation*}
			(-\Delta)_{\R}^s u(x):= (1-s) \pv \int\limits_{\R} \frac{u(x)-u(y)}{\abs{x-y}^{1+2s}} \d y.
		\end{equation*}
	\end{proposition}	
	\begin{remark}
		The constant $s^6 C$ in the previous proposition is not optimal in the limit $s\to 0+$.
	\end{remark}
	\begin{proof}
	Since $(\cdot)_+^s\zeta$ is locally $C^2$ and grows slower than $(1+\abs{\cdot})^{3s/2}$, the term $(-\Delta)_\R^s[(\cdot)_+^s\zeta]$ exists in $\R_+$. We fix $0<t<1$. Now, we use the change of variables $h=tr$ to find that
	\begin{align*}
		(-\Delta)_{\R}^s[(\cdot)_+^s\,\zeta](t)&=(1-s) \pv\int\limits_{\R}\frac{t^s\zeta(t)-(t+h)_+^s\zeta(t+h)}{\abs{h}^{1+2s}}\d h\\
		&=(1-s)\frac{\zeta(t)}{t^s} \pv\int\limits_{\R}\frac{1-(1+r)_+^s\frac{\zeta(t(1+r))}{\zeta(t)}}{\abs{r}^{1+2s}}\d r= \frac{\zeta(t)}{t^s}(-\Delta)_{\R}^sf_t (1),
	\end{align*}
	where the function $f_t$ is given by
	\begin{equation*}
		f_t(r)= (r)_+^s\frac{\zeta(tr)}{\zeta(t)}.
	\end{equation*}
	Just as in \cite{RoFe24}, we want to compare this function to a translated version of $(\cdot)_+^s$ which is $s$-harmonic in $\R_+$. Note that, if $\zeta(r)=r^\alpha$, then $f_t$ is independent of $t$. In our setup, this might not be the case, e.g.\ choose $\zeta(r)=1/\ln(1/r)^2$. More precisely, we choose $\kappa\ge 1$, $0<a<1$ such that $g_{a,\kappa}(r)= \kappa (r-a)_+^s$ touches $f_t$ at $r=1$ and is smaller for any $r\in \R_+\setminus\{ 1 \}$, see claim A. Thus, the constants $a$ and $\kappa$ must be chosen such that 
	\begin{align*}
		f_t(1)=g_{a,\kappa}(1), \qquad f_t'(1)=g_{a,\kappa}'(1).
	\end{align*}
	The first property is satisfied as soon as 
	\begin{align*}
		\kappa(1-a)^s = 1.
	\end{align*}
	The second one is true whenever 
	\begin{equation*}
		s+\frac{\zeta'(t)t}{\zeta(t)}= s\kappa (1-a)^{s-1}= s\kappa^{1/s-1}=s \frac{1}{1-a}.
	\end{equation*}
	We fix these choices, i.e.\ 
	\begin{align*}
		a := \frac{\zeta'(t)t}{s \zeta(t)+\zeta'(t)t},\quad
		\kappa:=\big( 1+\frac{\zeta'(t)t}{s\zeta(t)} \big)^{s}.
	\end{align*}
	Since $\zeta$ is concave and strictly increasing, we know that $t\zeta'(t)\le \zeta(t)$, $\zeta'>0$ and, thus, the inequalities
	\begin{equation}
		0<\frac{1}{s+1}\frac{t\zeta^\prime(t)}{\zeta(t)}\le a\le \frac{1}{s+1}<1
	\end{equation}
	 and $1< \kappa\le (1+1/s)^s$ hold.

	\textbf{Claim A.} With this choice of $a$ and $\kappa$, we have $g_{a,\kappa}\le f_t$ on $\R$.
	
	\textit{Proof of claim A.} If $r\le a$, then $g_{a,\kappa}(r)=0\le f_t(r)$. Since the two functions coincide at $r=1$, it remains to prove the $g_{a,\kappa}^\prime(r)\le f_t^\prime(r)$ for $r>1$ and $g_{a,\kappa}^\prime(r)\ge f_t^\prime(r)$ for $a<r<1$. We calculate 
	\begin{align*}
		f_t^\prime(r)&= s r^{s-1}\frac{\zeta(tr)}{\zeta(t)}+ r^{s-1}\frac{tr\zeta^\prime(tr)}{\zeta(t)},\\
		g_{a,\kappa}^{\prime}(r)&= s\kappa(r-a)^{s-1}= s r^{s-1} k(r)+\frac{t\zeta^{\prime}(t)}{\zeta(t)}r^{s-1} k(r),\\
		\intertext{where}
		k(r)&:= \Big( 1+(1-1/r)\frac{t\zeta^\prime(t)}{s\zeta(t)} \Big)^{s-1}.
	\end{align*}
	We notice a few things. If $r>1$, then $k(r)\le1$. If $r<1$, then $k(r)\ge 1$. Since $\zeta$ is increasing, we know $\zeta(tr)\ge \zeta(t)$ if $r>1$ and $\zeta(tr)\le \zeta(t)$ if $r<1$. This yields 
	\begin{equation}
		\begin{split}
			sr^{s-1}k(r)&\le sr^{s-1}\le s r^{s-1}\frac{\zeta(tr)}{\zeta(t)} \quad \text{ if }r>1,\\
			sr^{s-1}k(r)&\ge sr^{s-1}\ge s r^{s-1}\frac{\zeta(tr)}{\zeta(t)} \quad \text{ if }r<1.
		\end{split}
	\end{equation}
	By assumption \ref{ass:zeta_ndecr}, the map $x\mapsto x\zeta^\prime(x)$ is nondecreasing. Thus, we conclude 
	\begin{align*}
		\frac{t\zeta^{\prime}(t)}{\zeta(t)}r^{s-1} k(r)&\le \frac{tr\zeta^{\prime}(tr)}{\zeta(t)}r^{s-1} k(r)\le \frac{tr\zeta^{\prime}(tr)}{\zeta(t)}r^{s-1}\quad \text{ if }r>1,\\
		\frac{t\zeta^{\prime}(t)}{\zeta(t)}r^{s-1} k(r)&\ge \frac{tr\zeta^{\prime}(tr)}{\zeta(t)}r^{s-1} k(r)\ge \frac{tr\zeta^{\prime}(tr)}{\zeta(t)}r^{s-1}\quad \text{ if }r<1.
	\end{align*}
	These two observations yield $g_{a,\kappa}^\prime(r)\le f_t^\prime(r)$ for $r>1$ and $g_{a,\kappa}^\prime(r)\ge f_t^\prime(r)$ for $r<1$. This proves the claim. \medskip
	
	Claim A together with $g_{a,\kappa}\ne f_t$ immediately yields the negativity of $(-\Delta)_{\R}^s[(\cdot)_+^s\,\zeta](t)$. Since $f_t-g_{a,\kappa}$ has a global minimum at $1$ and $g_{a,\kappa}$ is $s$-harmonic in $(a,\infty)\ni 1$, we can write
	\begin{equation*}
		(-\Delta)_{\R}^s[(\cdot)_+^s\,\zeta](t)= \frac{\zeta(t)}{t^s}(-\Delta)^s[f_t](1)= \frac{\zeta(t)}{t^s}(-\Delta)^s[f_t-g_{a,\kappa}](1) <0.
	\end{equation*}
	In contrast to \cite{RoFe24}, this observation is not sufficient. Since $f_t$ may converge to $(\cdot)_+^s$ as $t\to 0+$, the previous inequality may become an equality in the limit $t\to 0$. Thus, we must analyze $(-\Delta)^s[f_t-g_{a,\kappa}](1)$ more carefully. 
	
		\textbf{Claim B.} There exists a constant $r_0>0$ such that for any $0<t<t_0$ and any $r\in (1,1+r_0)$
	\begin{equation}\label{eq:claim1}
		f_t(r)-g_{a,\kappa}(r)\ge \frac{c_0}{2} \frac{t\zeta^\prime(t)}{\zeta(t)}(1-r)^2.
	\end{equation}
	
	\textit{Proof of claim B.} First off, we assume $r_0<1$. From the construction of $g_{a,\kappa}$, we know that $h_t:=f_t-g_{a,\kappa}$ satisfies $h_t^\prime(1)=0=h_t(1)$. We expand the function $h_t$ at the point $r=1$. This yields
	\begin{equation*}
		h_t(r)=\frac{1}{2}h_t^{\prime\prime}(1)(r-1)^2 + \iota(r,t).
	\end{equation*}
	Here $\iota$ is the error of the approximation. Next, let's calculate $h_t^{\prime\prime}$. 
	\begin{align*}
		h_t^{\prime\prime}(1)&=s(s-1)+2s\frac{t\zeta^\prime(t)}{\zeta(t)}+\frac{t^2\zeta^{\prime\prime}(t)}{\zeta(t)}-s(s-1)\Big(\frac{s\zeta(t)+t\zeta^\prime(t)}{s\zeta(t)}\Big)^2\\
		&= 2\frac{t\zeta^\prime(t)}{\zeta(t)}+\frac{t^2\zeta^{\prime\prime}(1)}{\zeta(t)}+\frac{1-s}{s}\Big(\frac{t\zeta^\prime(t)}{\zeta(t)}\Big)^2\ge \frac{t\zeta^\prime(t)}{\zeta(t)}\Big( 2+\frac{t\zeta^{\prime\prime}(t)}{\zeta^\prime(t)} \Big)\ge \frac{t\zeta^\prime(t)}{\zeta(t)}.
	\end{align*}
	In the last equality we used \ref{ass:zeta_ndecr}. Now, we bound the error term. By Taylor's theorem, we know that there exists $z_r\in (1,r)$ such that
	\begin{align*}
		\iota(r,t)= \frac{h^{(3)}(z_r)}{6} (r-1)^3.
	\end{align*}
	We calculate the third derivative $h^{(3)}$: 
	\begin{align*}
		h^{(3)}(z_r)&= s(s-1)(s-2)z_r^{s-3}\frac{\zeta(tz_r)}{\zeta(t)}+ 3s(s-1)z_r^{s-2}\frac{t\zeta^\prime(tz_r)}{\zeta(t)}+ 3s z_r^{s-1}\frac{t^2\zeta^{\prime\prime}(tz_r)}{\zeta(t)}\\
		&\quad+z_r^{s}\frac{t^3\zeta^{(3)}(tz_r)}{\zeta(t)}-s(s-1)(s-2)\kappa(z_r-a)^{s-3}\\
		&= \I{1}+\I{2}+\I{3}+\I{4}+\I{5}.
	\end{align*}
	Note that the second term is easily bound using the monotonicity of $\zeta^\prime$ and $z_r>1$ by
	\begin{align*}
		\I{2}(r-1)^3/6\ge -\frac{s(1-s)}{2} 1^{s-2}r_0\frac{t\zeta^\prime(t)}{\zeta(t)}(r-1)^2.
	\end{align*}
	If we pick $r_0$ so small such that $2s(1-s)r_0<2^{-4}$, then this term may be estimated as in the claim.
	
	To estimate the third term, we use \eqref{eq:assumption_zeta_d_1}. This and again the monotonicity of $\zeta^\prime$ yield
	\begin{align*}
		\I{3}(r-1)^3/6 \ge -\frac{s}{2}1^{s-2}\frac{t\zeta^\prime(t)}{\zeta(t)}(r-1)^3.
	\end{align*}
	Here, we need to ensure that $r_0$ is so small such that $2sr_0<2^{-4}$. The term $\I{4}$ is estimated easily using the assumption \eqref{eq:assumption_zeta_d_1}. This together with the monotonicity of $\zeta^\prime$ yields
	\begin{align*}
		\I{4}(r-1)^3/6\ge -1^{s-2}\frac{c_1}{6} \frac{t\zeta^\prime(tz_r)}{\zeta(t)}(r-1)^3\ge-\frac{c_1r_0}{6} \frac{t\zeta^\prime(t)}{\zeta(t)}(r-1)^2.
	\end{align*}
	Thus, we pick $r_0>0$ sufficiently small such that $c_1 \,r_0 /6\le 2^{-5}$.
	Now, we consider the terms $\I{1}$ and $\I{5}$. Note that these terms admit additional cancellation. Using the monotonicity of $\zeta$, the convexity of $x\mapsto x^{s-3}$, and $(1+x)^s\le 1+sx$, we find 
	\begin{align*}
		\I{1}+\I{5}&\ge s(1-s)(2-s)\Big( z_r^{s-3}-(z_r-a)^{s-3}+(1- \kappa)(z_r-a)^{s-3} \Big)\\
		&\ge s(1-s)(2-s)\Big( -(3-s)(z_r-a)^{s-4}a -s\frac{t\zeta^\prime(t)}{s\zeta(t)}(z_r-a)^{s-3} \Big)\\
		&\ge -(1-s)(2-s)\frac{3 \,s^{s-3}}{(1+s)^{s-3}} \frac{t\zeta^\prime(t)}{\zeta(t)}.
	\end{align*}
	If we ensure that $(1-s)(2-s)3s^{s-3}/(1+s)^{s-3}r_0<2^{-5}$, then we find 
	\begin{align*}
		\iota(r,t)\ge - 2^{-4} \frac{t\zeta^\prime(t)}{\zeta(t)}(r-1)^2
	\end{align*}
	which proves the claim. \medskip

	Using this claim, we easily deduce an upper bound on $(-\Delta)^s_{\R}[(\cdot)_+^s \zeta](t)$. Again, since $f_t-g_{a,\kappa}$ has a global minimum at $1$, see claim A, we write for any $t<t_0$
	\begin{equation*}
		(-\Delta)^s_{\R}[(\cdot)_+^s \zeta](t)\le -\frac{\zeta(t)}{t^s} (1-s)\int\limits_{1}^{1+r_0} \frac{f_t(r)-g_{a,\kappa}(r)}{\abs{1-r}^{1+2s}}\d r\le - \frac{c_0}{2} r_0^{2-2s}\frac{t\zeta^\prime(t)}{t^s}.
	\end{equation*}
	In the last inequality we used claim B. 
	\end{proof}

	\section{Barriers in higher dimensions}\label{sec:higher_dimensional}
	In this section, we construct explicit super and sub-solution which we will use as barrier functions in the proof of \autoref{prop:distance_lower_bound_with_pUB}, \autoref{prop:distance_lower_bound_without_pUB}, \autoref{prop:distance_upper_bound_with_pUB}, and \autoref{prop:distance_upper_bound_without_pUB}. \smallskip
	
	We fix a uniform $\cdini$-domain (respectively $2s$-$\cdini$-domain) $D \subset \R^d$ with a localization radius $\rho>0$ and a modulus of continuity $\tilde{\omega}$, see \autoref{def:c1dini_domain}. Recall that the map $\tilde{\omega}:[0,\infty)\to [0,\infty)$ is continuous, nondecreasing, sublinear, i.e.\ $\limsup_{t\to \infty}\tilde{\omega}(t)/t\le M$, $\tilde{\omega}(0)=0$, and satisfies \eqref{eq:dini}.
	
	We need the following modification of $\tilde{\omega}$ which is slightly bigger and satisfies some further properties which are needed in the following estimates, see also \autoref{rem:replace_modulus}.
	
	\begin{lemma}\label{lem:modulus}
		Let $0<\iota<1/3$ and $\overline{\omega}:[0,\infty)\to (0,\infty)$ a nondecreasing, continuous, sublinear function satisfying \eqref{eq:dini} respectively \eqref{eq:dini_2s}. There exists a continuous, increasing, and concave function $\omega:[0,\infty)\to [0,\infty)$, $\omega\in C^2((0,\infty))$ such that $\omega\ge \overline{\omega}$, and $\omega$ satisfies \eqref{eq:dini} respectively \eqref{eq:dini_2s}. Furthermore, there exists $t_0>0$ such that 
		\begin{align}
			\frac{\omega(t)}{t^{\iota}} \text{ is decreasing in }(0,t_0),\label{eq:modulus_decreasing_iota}\\
			t^2\omega^{\prime\prime}(t)\ge- \omega(t)-3 t\omega^{\prime}(t)\quad \text{ for all }0<t<t_0.\label{eq:modulus_second_derivative}
		\end{align}
	\end{lemma}
	\begin{proof}
		\textit{Step 1. (Concavity of $\omega$)} By the sublinearity of $\overline{\omega}$, we fix two constants $M, t_0>0$ such that $\overline{\omega}(t)\le M t$ for $t\ge t_0$. Without loss of generality, we assume $t_0<1/2$ while, if needed, enlarging the constant $M$ slightly. Now, set 
		\begin{align*}
			\omega_1(t)= \inf\{ \overline{\omega}(\delta)+\frac{Mt_0 t}{\delta}  \mid 0<\delta<t_0\}.
		\end{align*}
		Next, we prove that $\omega_1$ is not smaller than $\overline{\omega}$. Let $0<\delta<t_0$ and $t>0$ be arbitrary. If $0<t\le \delta$, then $\overline{\omega}(t)\le \overline{\omega}(\delta)\le \overline{\omega}(\delta)+Mt_0t/\delta$ since $\overline{\omega}$ is nondecreasing. If $t\ge t_0$, then $\overline{\omega}(t)\le M t\le Mt_0t/\delta \le \overline{\omega}(\delta)+Mt_0t/\delta$ by the choice of $M$ and $t_0$. If $\delta<t< t_0$, then 
		\begin{equation*}
			\overline{\omega}(t)\le \overline{\omega}(t_0)\le Mt_0\le Mt_0t/\delta\le \overline{\omega}(\delta)+ Mt_0t/\delta.
		\end{equation*} 
		Since $\delta$ was arbitrary, we find $\overline{\omega}(t)\le \omega_1(t)$. \smallskip
		
		Obviously, the function $\omega_1$ is also nondecreasing, continuous, and satisfies $\omega_1(0)=0$. Furthermore, $\omega_1$ is concave since 
		\begin{align*}
			&\omega_1(\lambda x + (1-\lambda)y)= \inf\{ \overline{\omega}(\delta)+\frac{Mt_0}{\delta}(\lambda x + (1-\lambda)y)  \mid 0<\delta<t_0\}\\
			&\qquad\qquad\ge \lambda\inf\{ \overline{\omega}(\delta)+\frac{Mt_0}{\delta} x  \mid 0<\delta<t_0\}+(1-\lambda)\inf\{ \overline{\omega}(\delta)+\frac{Mt_0}{\delta} y  \mid 0<\delta<t_0\}.
		\end{align*}
		Next, we prove that $\omega_1$ satisfies \eqref{eq:dini}. For this, we use the trivial bound
		\begin{equation*}
			\omega_1(t)\le \overline{\omega}(t\ln(1/t)^2)+ Mt_0/\ln(1/t)^2
		\end{equation*}
		for any $0<t<t_0$. Firstly, the second term is integrable with $1/t$ near zero. Secondly, using the substitution $y=t\ln(1/t)^2$, we find
		\begin{align*}
			\int\limits_{0}^{e^{-4}} \frac{\overline{\omega}(t\ln(1/t)^2)}{t}\d t\le 2\int\limits_{0}^{4^2/e^{4}}\frac{\overline{\omega}(y)}{y}\d y<\infty. 
		\end{align*}
		One checks that $\omega_1$ satisfies \eqref{eq:dini_2s} using the same  bound and the same change of variables. Now, we regularize $\omega_1$ by taking averaged means, say $\omega_2(t)=\fint_t^{2t}\omega_1(r)\d r$. This function inherits all previous properties of $\omega_1$. \medskip
		
		\textit{Step 2. (\eqref{eq:modulus_decreasing_iota})} We define the function $j_{\iota}:[0,\infty)\to [0,\infty)$ via 
		\begin{equation*}
			j_\iota(t):= (1-\iota)t^{\iota}\int\limits_{t}^{\infty}\frac{\min\{ \omega_2(r),\omega_2(1) \}}{r^{1+\iota}}\d r.
		\end{equation*}
		This function satisfies for any $t<1$
		\begin{equation}\label{eq:j_iota_correction_function}
			tj_\iota^{\prime}(t)= \iota j_{\iota}(t)-(1-\iota)\omega_2(t).
		\end{equation}
		Since $\omega_2$ is nondecreasing, we also know
		\begin{equation*}
			j_\iota(t)\ge (1-\iota)t^{\iota}\int\limits_{t}^\infty \frac{\omega_2(t)}{r^{1+\iota}}\d r = \frac{1-\iota}{\iota}\omega_2(t).
		\end{equation*}
		Plugging this bound into \eqref{eq:j_iota_correction_function} reveals that $j_\iota$ is nondecreasing. Next, we prove that $j_{\iota}$ is also concave. Firstly, $j_\iota\in C([0,\infty))\cap C^3_{\loc}((0,\infty))$ is constant in $(1,\infty)$ and nondecreasing in $\R_+$. It remains to prove that $j_{\iota}^{\prime\prime}(t)\le 0$ for $t<1$. We calculate
		\begin{align*}
			t^2j_{\iota}^{\prime\prime}(t)&=\iota (tj_{\iota}^\prime(t)-j_{\iota}(t))-(1-\iota)(t\omega_2^{\prime}(t)-\omega_2(t))\\
			&=-\iota(1-\iota)j_{\iota}(t)-(1-\iota)\omega_2^{\prime}(t)t+ (1-\iota)^2\omega_2(t)\le -(1-\iota)\omega_2^{\prime}(t)t.
		\end{align*} 
		The function $j_{\iota}$ also satisfies \eqref{eq:dini}. This follows from 
		\begin{align*}
			\int\limits_{0}^{1}\frac{j_{\iota}(t)}{t}\d t &= (1-\iota)\int\limits_{0}^{1}t^{\iota-1}\int\limits_{t}^{\infty}\frac{\min\{ \omega_2(r),\omega_2(1) \}}{r^{1+\iota}}\d r\d t= (1-\iota)\Big(\frac{\omega_2(1)}{\iota^2} + \int\limits_{0}^{1}t^{\iota-1}\int\limits_{t}^{1}\frac{ \omega_2(r)}{r^{1+\iota}}\d r\d t\Big)\\
			&=(1-\iota)\frac{\omega_2(1)}{\iota^2} + \frac{1-\iota}{\iota}\int\limits_{0}^{1} \frac{ \omega_2(r)}{r}\d r<\infty.
		\end{align*}
		Similarity, one proves that \eqref{eq:dini_2s} is satisfied by $j_{\iota}$. This allow us to define a new function $\omega_3$ which is slightly bigger than $\omega_2$, is still nondecreasing and concave, and additionally satisfies \eqref{eq:modulus_decreasing_iota}. Set $\omega_3:= \omega_2+j_{\iota}$. Then, for $t<1$
		\begin{align*}
			t\omega_3^{\prime}(t)= t\omega_2^{\prime}(t)+ tj_{\iota}^{\prime}(t)= \iota(\omega_2(t)+j_\iota(t))+ t\omega_2^{\prime}(t)-\omega_2(t)\le \iota\omega_3(t).
		\end{align*}
		Here, we used the concavity of $\omega_2$. Therefore, $\omega_3(t)/t^{\iota}$ is decreasing in $(0,1)$. \smallskip
		
		\textit{Step 3. (\eqref{eq:modulus_second_derivative})} We regularize $\omega_3$ by taking averaged means twice. Set $c=1/(1-3\iota)$. We define
		\begin{align*}
			\omega_4(t):= c\fint_{t/2}^t\fint_{r/2}^r \omega_3(l)\d l \d r.
		\end{align*}
		Then, we use the monotonicity of $\omega_3$ and \eqref{eq:modulus_decreasing_iota} to estimate $\omega_4(t)$ from below by $\omega_3(t)$, i.e.\
		\begin{align*}
			\omega_4(t)\ge c\omega_3(t/4)\ge c(\omega_3(t)-\omega_3^{\prime}(t/4)(3t/4)  )\ge c(\omega_3(t)-3\iota\omega_3(t/4))\ge c(1-3\iota)\omega_3(t)=\omega_3(t).
		\end{align*}
		A minor calculation reveals
		\begin{align*}
			\omega_4(t)= 4c\int\limits_{1/2}^1\int\limits_{1/2}^1 \omega_3(tlr)\d l \d r.
		\end{align*}
		From this representation, it is clear that $\omega_4$ inherits all of the previous properties of $\omega_3$. We differentiate $\omega_4$ 
		\begin{equation*}
			t\omega_4^{\prime}(t)=4c\int\limits_{1/2}^1\int\limits_{1/2}^1l\partial_l\omega_3(tlr)\d l \d r=4c\int\limits_{1/2}^1\Big( \omega_3(tr)-\frac{1}{2}\omega_3(tr/2)- \int\limits_{1/2}^1\omega_3(tlr)\d l  \Big)\d r.
		\end{equation*}
		Thus, the term $t\omega_4^{\prime}(t)+t^2\omega_4^{\prime\prime}(t)$ equals
		\begin{align*}
			t\partial_t t\omega_4^{\prime}(t)&=4c\int\limits_{1/2}^1\Big( r\partial_r\omega_3(tr)-\frac{r}{2}\partial_r\omega_3(tr/2)- \int\limits_{1/2}^1l\partial_l\omega_3(tlr)\d l  \Big)\d r\\
			&=4c\Bigg(\omega_3(t)-\frac{1}{2}\omega_3(t/2)-\int\limits_{1/2}^1 \omega_3(tr)\d r - \frac{1}{2}\omega_3(t/2)+\frac{1}{4}\omega_3(t/4)+\frac{1}{2}\int\limits_{1/2}^1 \omega_3(tr/2)\d r\\
			&\qquad -\int\limits_{1/2}^1\Big( \omega_3(tr)-\frac{1}{2}\omega_3(tr/2)-\int\limits_{1/2}^1 \omega_3(tlr)  \d l \Big)\d r\Bigg)\\
			&=4c\Bigg(\omega_3(t)-\frac{1}{2}\omega_3(t/2)-\int\limits_{1/2}^1 \omega_3(tr)\d r - \frac{1}{2}\omega_3(t/2)+\frac{1}{4}\omega_3(t/4)+\frac{1}{2}\int\limits_{1/2}^1 \omega_3(tr/2)\d r\Bigg)\\
			&\quad -t\omega_4^{\prime}(t)\\
			&\ge 4c\Big( -\frac{1}{4}\frac{t\omega_3^{\prime}(t/4)}{4}- \frac{\omega_3(t/2)-\omega_3(t/4)}{4}+\frac{3}{8}t\omega_3^{\prime}(t) \Big) -t\omega_4^{\prime}(t)\\
			&\ge -c\Big(\omega_3(t/2)-3/2 t\omega_3^{\prime}(t)  \Big) - t\omega_4^{\prime}(t).
		\end{align*}
		Now, we use $\omega_3(t/2)\le 1/c \omega_4(2t)\le 1/c (\omega_4(t)+\omega_4^{\prime}(t)t)$. This yields
		\begin{align*}
		t^2\omega_4^{\prime\prime}(t) \ge -c\Big(\omega_3(t/2)-3/2 t\omega_3^{\prime}(t)  \Big) - 2t\omega_4^{\prime}(t)\ge- \omega_4(t)-3 t\omega_4^{\prime}(t).
		\end{align*}		
		\textit{Step 4.} We set $\omega(t):=\omega_4(t)+t_{+}^{\iota/2}$ which is increasing.
	\end{proof}
	\begin{remark}\label{rem:replace_modulus}
		We apply \autoref{lem:modulus} to the modulus of continuity $\tilde{\omega}$ which yields a function $\omega$. Since $\omega\ge \tilde{\omega}$, we will assume without loss of generality that the modulus of continuity of the $\cdini$-domain (respectively $2s$-$\cdini$-domain) $D$ satisfies all properties of $\omega$. 
	\end{remark}
	
	Adapted to $D$, we pick a regularized distance function $\fd\in C^{0,1}(\R^d)\cap C^{2}_{\loc}((\partial D)^c)$ which satisfies 
	\begin{equation}\label{eq:reg_distance_properties}
		\begin{split}
			C_1^{-1} d_D(x)\le \fd(x)&\le C_1 d_D(x),\\
			\abs{\nabla \fd(x)-\nabla \fd(y)}&\le C_2 \omega(\abs{x-y}),\\
			\abs{D^2 \fd(x)}&\le C_3 \frac{\omega(d_D(x))}{d_D(x)}.
		\end{split}
	\end{equation}
	This exists by \cite{Lie85}. Recall that $d_D(x)= \distw{x}{D^c}$.\smallskip
	
	The next lemma is an essential technical estimate used in both \autoref{prop:almost_harmonic} and \autoref{prop:subsolution_d}. It is needed whenever \eqref{ass:pUB} is not available. In contrast, if \eqref{ass:pUB} holds, then we can use \cite[Lemma B.2.4]{RoFe24} instead.

	\begin{lemma}\label{lem:intermediate_estimate_1}
		Let $s_0\in (0,1)$, $a\in (0,1/4)$, and $D$ be a uniform $\cdini$-domain with a modulus of continuity $\omega$ as in \autoref{rem:replace_modulus} that satisfies all properties from \autoref{lem:modulus} with $\iota=\min\{1/6, s_0/4\}$ and a localization radius $\rho$. There exist $0<\rho_1<\rho$ and a constant $C= C(s_0, \omega, \rho, a)>0$ such that for any $x\in D$ with $d_D(x)<\rho_1/2$ and $\theta\in S^{d-1}$ we have
		\begin{align*}
			\int\limits_{ad_D(x)}^{\rho_1} &\big|\fd(x+r\theta)^s- (\fd(x)+\nabla \fd(x)\cdot r\theta)_+^s\big|\frac{\d r}{r^{1+2s}}\\
			&\qquad\le C \frac{\omega(f^{-1}(d_D(x)))}{d_D(x)^s} \begin{cases}
				\frac{(\omega(f^{-1}(d_D(x))))^{2s-1}- 1}{1-2s}&, s\ne 1/2 \\
				\ln\big( \frac{1}{\omega(f^{-1}(d_D(x)))} \big)&, s=1/2
			\end{cases}, 
		\end{align*}
		where $f(t) = t \omega(t)$.
	\end{lemma}
	Note that $f$ is invertible since $t\mapsto t\omega(t)$ is strictly increasing and $\lim_{t\to \infty}t\omega(t)=\infty$. \smallskip 
	 
	Before we proceed with the lengthy and technical proof of this estimate, we give a remark on its importance. 
	\begin{remark}
		Using the bound 
		\begin{equation*}
			\big|\fd(x+r\theta)^s- (\fd(x)+\nabla \fd(x)\cdot r\theta)_+^s\big|\le C \abs{r}^s \omega(r)^s
		\end{equation*}
		as in the proof of \cite[Proposition B.2.1 (i)]{RoFe24}, we get, quite easily, 
		\begin{equation}\label{eq:intermediate_worse_estimate}
			\int\limits_{ad_D(x)}^{\rho_1} \big|\fd(x+r\theta)^s- (\fd(x)+\nabla \fd(x)\cdot r\theta)_+^s\big|\frac{\d r}{r^{1+2s}}\le C \frac{\omega(d_D(x))^s}{d_D(x)^s}
		\end{equation}
		which is a weaker estimate than the one stated in \autoref{lem:intermediate_estimate_1}. In particular, using only the bound \eqref{eq:intermediate_worse_estimate} instead of \autoref{lem:intermediate_estimate_1} would require 
		\begin{equation}\label{eq:bad_s_dini}
			\int\limits_{0}^1 \frac{\omega(t)^s}{t}\d t<\infty
		\end{equation}
		in order to construct appropriate barriers in \autoref{prop:subharmonic_barrier_without_pUB} and \autoref{prop:supharmonic_barrier_without_pUB}. But the assumption \eqref{eq:bad_s_dini} on the modulus of continuity $\omega$ would be much stronger than the Dini-continuity \eqref{eq:dini}, e.g.\ consider $\omega(t)=1/\ln(1/t)^{(1+s)/(2s)}$. In \cite{RoFe24}, this is no obstruction as they deal with $C^{1,\alpha}$-domains those modulus of continuity can be chosen as $\omega(t)=c\, t^\alpha$. 
	\end{remark}
	\begin{proof}[Proof of \autoref{lem:intermediate_estimate_1}]
		Since the problem is translation and rotation invariant, we assume the point $x=(x',x_d)\in D$ in the domain $D$ is of the form $x= (0,x_d)$ and that $z=0\in \partial D$ is the minimizer of $x$ to the boundary $\partial D$. Note that $d_D(x)$ simplifies to $x_d$. Since $D$ is locally the epigraph of a $\cdini$-function, there exists $\phi$ such that $D\cap B_\rho(0)= \{ (y',y_d)\in B_\rho(0)\mid y_d>\phi(y') \}$ and $\phi(0)=0=\abs{\nabla\phi(0)}$. The function $\phi$ satisfies 
		\begin{equation}\label{eq:phi_properties}
			\begin{split}
				\abs{\phi(y')}&\le \abs{\phi(y')-\phi(0)-y'\cdot \nabla\phi(0)}\le \abs{y'}\omega(\abs{y'}),\\
				\abs{\nabla\phi(y')}&= \abs{\nabla \phi(y')-\nabla \phi(0)}\le \omega(\abs{y'})
			\end{split}
		\end{equation}
		for any $y'\in \R^{d-1}$. Now, fix any $\theta \in S^{d-1}$. We divide the integration domain into two parts
		\begin{align*}
			\Big(\int\limits_{f^{-1}(ax_d)}^{\rho_1}+	\int\limits_{ax_d}^{f^{-1}(ax_d)}\Big) \big|\fd(x+r\theta)^s- (\fd(x)+\nabla \fd(x)\cdot r\theta)_+^s\big|\frac{\d r}{r^{1+2s}}:= \I{}+\II{}.
		\end{align*}
		To estimate $\I{}$, we use the bound 
		\begin{equation}\label{eq:taylor_distance}
			|\fd(x+y)-(\fd(x)+\nabla \fd(x)\cdot y)_+| \le \omega(\abs{y})\abs{y},
		\end{equation}
		see \cite[Lemma B.2.3.]{RoFe24}. This and the concavity of $(\cdot)^s$ yield
		\begin{align*}
			\I{}\le \int\limits_{f^{-1}(ax_d)}^{\rho_1}\omega(r)^s\frac{\d r}{r^{1+s}}\le \frac{\omega(f^{-1}(ax_d))^s}{f^{-1}(ax_d)^{s\iota}}\int\limits_{f^{-1}(ax_d)}^{\infty}\frac{\d r}{r^{1+s(1-\iota)}}= \frac{1}{s(1-\iota)}\frac{\omega(f^{-1}(ax_d))^s}{f^{-1}(ax_d)^{s}}.
		\end{align*}
		Here, we used \eqref{eq:modulus_decreasing_iota}. Since $f(t)=t \omega(t)$, we know 
		\begin{equation}\label{eq:property_of_f}
			t= f^{-1}(t)\,\omega(f^{-1}(t)).
		\end{equation} 
		Because of this and since $\iota<1/3$, the estimate of  $\I{}$ reads
		\begin{align*}
			\I{}\le \frac{3}{2s}\frac{\omega(f^{-1}(ax_d))^{2s}}{a^s \,x_d^{s}}\le \frac{3}{a2s} \frac{\omega(f^{-1}(x_d))^{2s}}{x_d^s}.
		\end{align*}
		Here, we used the monotonicity of both $\omega$ and $f^{-1}$. We notice that this implies the desired estimate if we assume $\rho_1$ to be so small such that $\omega(f^{-1}(x_d))\le \min\{1/e^2, s/2\}$. If $s= 1/2$, then this is obvious since $\ln(1/\omega(f^{-1}(x_d)))\ge \ln(e^2)=2$. If $s\ne 1/2$, then it is due to 
		\begin{equation}\label{eq:lem_techincal_estimate_final_bound_small_s}
			a^{2s-1}\le \frac{a^{2s-1}-1}{1-2s}
		\end{equation}
		for any $a\le s/2\le (2s)^{1/(1-2s)}$. 
	
		Now, we turn our attention to the more delicate term $\II{}$. The concavity of $(\cdot)^s$ and \eqref{eq:taylor_distance} yield
		\begin{align*}
			\II{}\le s 	\int\limits_{ax_d}^{f^{-1}(ax_d)} \big((\fd(x)+\nabla \fd(x)\cdot r\theta)_+^{s-1}+ \fd(x+r\theta)^{s-1}\big)\frac{\omega(r)}{r^{2s}}\d r= \II{1}+\II{2}.
		\end{align*}
		Here, we use the convention $(0)_+^{s-1}=0$. We begin by estimating the term $\II{1}$. 
		
		\textbf{Estimate of $\II{1}$.}	We distinguish two cases. 
		
		\textit{Case 1.} If $\nabla \fd(x)\cdot \theta \ge 0$, then for any $r>0$
		\begin{equation*}
			(\fd(x)+\nabla \fd(x)\cdot r\theta)_+^{s-1}\le \fd(x)^{s-1}\le C_1^{1-s}\, x_d^{s-1}.
		\end{equation*}
		This observation yields
		\begin{align*}
			\II{1}&\le sC_1^{1-s}x_d^{s-1} 	\int\limits_{ax_d}^{f^{-1}(ax_d)} \frac{\omega(r)}{r^{2s}}\d r.\\
			\intertext{After scaling $r$ by $a x_d$ and using the monotonicity of $\omega$, this equals }
			& sC_1^{1-s}x_d^{-s} a^{1-2s}	\int\limits_{1}^{\frac{f^{-1}(ax_d)}{a x_d}} \frac{\omega(a x_d r)}{r^{2s}}\d r\le s\frac{C_1}{a}\frac{1-a}{a} \frac{\omega(f^{-1}(ax_d)) 	}{x_d^{s}}\int\limits_{1}^{\frac{1}{\omega(f^{-1}(ax_d)}} \frac{1}{r^{2s}}\d r.
		\end{align*}
		In the last inequality, we used \eqref{eq:property_of_f} to rewrite the upper boundary of the integration domain. This is the desired estimate for the term $\II{1}$ in the case $\nabla \fd(x)\cdot \theta \ge 0$. \smallskip
		
		\textit{Case 2.} If $\nabla\fd(x)\cdot \theta<0$, then the integrand of $\II{1}$ is only nontrivial if $r< \fd(x)/ \abs{\nabla \fd(x)\cdot \theta}$. In this case, we estimate $\II{1}$ as follows
		\begin{align*}
			\II{1}&\le s \int\limits_{ax_d}^{f^{-1}(ax_d)\wedge \frac{\fd(x)}{\abs{\nabla \fd(x)\cdot \theta}}} \big( \fd(x)-r \abs{\nabla \fd(x)\cdot \theta} \big)^{s-1} \frac{\omega(r)}{r^{2s}}\d r= \II{1,1}+\II{1,2}.
		\end{align*}
		Here, the term $\II{1,1}$ is defined as
		\begin{align*}
			\II{1,1}&:= s \int\limits_{ax_d}^{f^{-1}(ax_d)/2\wedge \frac{\fd(x)}{2\abs{\nabla \fd(x)\cdot \theta}}} \big( \fd(x)-r \abs{\nabla \fd(x)\cdot \theta} \big)^{s-1} \frac{\omega(r)}{r^{2s}}\d r\\
			&\le s\fd(x)^{s-1}2^{1-s} \int\limits_{ax_d}^{f^{-1}(ax_d)/2\wedge \frac{\fd(x)}{2\abs{\nabla \fd(x)\cdot \theta}}} \frac{\omega(r)}{r^{2s}}\d r \le s\frac{(2C_1)^{1-s}}{a^{2s-1}} \frac{\omega(f^{-1}(ax_d))}{x_d^{s}}\int\limits_{1}^{\frac{1}{\omega(f^{-1}(ax_d))}} \frac{1}{r^{2s}}\d r.
		\end{align*}
		As in the bound for the term $\II{1}$ in the case 1, this is the desired estimate. Furthermore, the term $\II{1,2}$ equals
		\begin{equation}\label{eq:estimate_II_12}
			\begin{split}
				s &\int\limits_{f^{-1}(ax_d)/2\wedge \frac{\fd(x)}{2\abs{\nabla \fd(x)\cdot \theta}}}^{f^{-1}(ax_d)\wedge \frac{\fd(x)}{\abs{\nabla \fd(x)\cdot \theta}}} \big( \fd(x)-r \abs{\nabla \fd(x)\cdot \theta} \big)^{s-1} \frac{\omega(r)}{r^{2s}}\d r\le  \frac{\omega\big(f^{-1}(ax_d)\wedge \frac{\fd(x)}{\abs{\nabla \fd(x)\cdot \theta}}\big)}{\big(f^{-1}(ax_d)/2\wedge \frac{\fd(x)}{2\abs{\nabla \fd(x)\cdot \theta}}\big)^{2s}}  \\
				&\times \frac{\Big(  \big(\fd(x)-\frac{f^{-1}(ax_d)\abs{\nabla \fd(x)\cdot \theta}}{2}\big)\vee \frac{\fd(x)}{2} \Big)^{s}-\Big(  \big(\fd(x)-f^{-1}(ax_d)\abs{\nabla \fd(x)\cdot \theta}\big)\vee 0 \Big)^{s} }{\abs{\nabla \fd(x)\cdot \theta}}.
			\end{split}
		\end{equation}
		In the previous estimate, we used that $\omega$ is increasing, $r\mapsto r^{-2s}$ is decreasing, bounded them using the integration domain, and integrated the remaining term. We distinguish two more cases. \smallskip
		
		\textit{Sub-case 2.1.} If $\fd(x)/ \abs{\nabla \fd(x)\cdot \theta}\le f^{-1}(ax_d)$, then the previous bound reduces to 
		\begin{equation}\label{eq:help_bound_II_1_2}
			\II{1,2}\le \omega(\frac{\fd(x)}{\abs{\nabla \fd(x)\cdot \theta}})2^{1-s} \big( \frac{\fd(x)}{\abs{\nabla \fd(x)\cdot \theta}} \big)^{1-2s}\fd(x)^{s-1}.
		\end{equation}
		Note that $a x_d\le \fd(x)/\abs{\nabla \fd(x)\cdot \theta}\le f^{-1}(ax_d)$ since otherwise the term $\II{1,2}$ would not yield a contribution to $\II{1}$. 
		
		If $s\le 1/2$, then $r\mapsto \omega(r)\,r^{1-2s}$ is increasing and, thus, using the above bound and \eqref{eq:property_of_f} yield 
		\begin{align*}
			\II{1,2}&\le  2C_1 x_d^{s-1}  \omega(f^{-1}(ax_d)) \big( f^{-1}(ax_d) \big)^{1-2s}\\
			&= 2C_1 \frac{\omega(f^{-1}(ax_d))^{2s}}{x_d^{s}}  \Big(\frac{\omega(f^{-1}(ax_d)) f^{-1}(ax_d)}{x_d} \Big)^{1-2s}=2C_1 \frac{\omega(f^{-1}(ax_d))^{2s}}{x_d^s}.
		\end{align*} 
		By \eqref{eq:lem_techincal_estimate_final_bound_small_s}, the desired bound follows. 
		
		If $s>1/2$, then using \eqref{eq:help_bound_II_1_2} we find
		\begin{align*}
			\II{1,2}\le \omega(f^{-1}(ax_d))2^{1-s} \big( ax_d \big)^{1-2s}\fd(x)^{s-1}\le C_1 2^{1-s} (a\wedge 1)^{-1}\frac{\omega(f^{-1}(x_d)) }{x_d^{s}}.
		\end{align*}
		This is the desired estimate for $\II{1,2}$ in the case  $\fd(x)/ \abs{\nabla \fd(x)\cdot \theta}\le f^{-1}(ax_d)$. \smallskip
		
		\textit{Sub-case 2.2.} If $\fd(x)/\abs{\nabla \fd(x)\cdot \theta}>f^{-1}(a x_d)$, then we bound $\II{1,2}$ using \eqref{eq:estimate_II_12} by 
		\begin{align*}
			\frac{\omega\big(f^{-1}(ax_d)\big)}{\big(f^{-1}(ax_d)/2\big)^{2s}}  \frac{  \big(\fd(x)-\frac{f^{-1}(ax_d)\abs{\nabla \fd(x)\cdot \theta}}{2}\big)^{s}-  \big(\fd(x)-f^{-1}(ax_d)\abs{\nabla \fd(x)\cdot \theta}\big)^{s} }{\abs{\nabla \fd(x)\cdot \theta}}.
		\end{align*} 
		\textit{Sub-case 2.2.1.} If additionally $2f^{-1}(a x_d)< \fd(x)/\abs{\nabla \fd(x)\cdot \theta}$, then $\II{1,2}$ is smaller than 
		\begin{align*}
			&s\frac{\omega\big(f^{-1}(ax_d)\big)}{\big(f^{-1}(ax_d)/2\big)^{2s}}  \frac{ \big(\fd(x)-f^{-1}(ax_d)\abs{\nabla \fd(x)\cdot \theta}\big)^{s-1} f^{-1}(ax_d)\abs{\nabla \fd(x)\cdot \theta} }{2\abs{\nabla \fd(x)\cdot \theta}}\\
			&\quad\le s\frac{\omega\big(f^{-1}(ax_d)\big)}{\big(f^{-1}(ax_d)/2\big)^{2s}}  2^{-s}\fd(x)^{s-1} f^{-1}(ax_d)\le sa^{1-2s} 2^s C_1^{1-s} \frac{\omega(f^{-1}(x_d))^{2s}}{x_d^{s}}.
		\end{align*}
		Here, we used \eqref{eq:estimate_II_12}, \eqref{eq:property_of_f}, \eqref{eq:reg_distance_properties}, and $t^s-r^s\le sr^{s-1}(t-r)$, i.e.\ the concavity of $(\cdot)^s$. 
		
		\textit{Sub-case 2.2.2.} If, additionally, $2f^{-1}(a x_d)\ge  \fd(x)/\abs{\nabla \fd(x)\cdot \theta}$, then $\II{1,2}$ is smaller than 
		\begin{align*}
			\frac{\omega\big(f^{-1}(ax_d)/2\big)}{\big(f^{-1}(ax_d)/2\big)^{2s}} & \frac{  \fd(x)^{s}}{\abs{\nabla \fd(x)\cdot \theta}}\le \frac{\omega\big(f^{-1}(ax_d)/2\big)}{\big(f^{-1}(ax_d)/2\big)^{2s}}    \fd(x)^{s-1} 2f^{-1}(ax_d)\le 2^{1+2s}  \frac{a C_1^{1-s} x_d^s}{f^{-1}(ax_d)^{2s}}\\
			&= 2^{1+2s} a^{1-2s} C_1^{1-s} \frac{\omega(f^{-1}(ax_d))^{2s}}{x_d^s}\le 2^{1+2s} a^{1-2s} C_1^{1-s} \frac{\omega(f^{-1}(x_d))^{2s}}{x_d^s}.
		\end{align*}
		Here, we used \eqref{eq:property_of_f}, the bound on the regularized distance \eqref{eq:reg_distance_properties}, and the monotonicity of $\omega$ and $f^{-1}$. This proves the desired bound on the term $\II{1,2}$ in all cases. \smallskip
		
		We successfully estimated the term $\II{1}$. Now, we consider the term $\II{2}$. \medskip
		
		\textbf{Estimate of $\II{2}$.} We distinguish two cases depending on the angle $\theta$
		
		\textit{Case 1.} The line segment $\{x+r\theta\mid r\in (0,f^{-1}(ax_d))\}$ hits the boundary $\partial D$. 
		
		We call the smallest $r$ such that $x+r\theta \in \partial D$ simply $r_0\in [x_d, f^{-1}(ax_d)]$. Now, we prove a three intermediate results. \smallskip
		
		\textit{Claim a.} There exists a global constant $c_1>0$ such that the distance to the boundary of $D$ is bounded from below by a fraction of the vertical distance, i.e.\ $\distw{y}{\partial D}\ge c_1 \abs{y_d-\phi(y')}$.
		
		The claim a is standard. We provide the proof for self-containment. We estimate
		\begin{align*}
			\abs{y_d-\phi(y')}&\le \inf\limits_{\tilde{y}'} \abs{y_d-\phi(\tilde{y}')}+\abs{\phi(y')-\phi(\tilde{y}')}\\
			&\le \sqrt{2}(1+[\phi]_{C^{0,1}})\inf\limits_{\tilde{y}'} \norm{y-(\tilde{y}', \phi(\tilde{y}'))}= \sqrt{2}(1+[\phi]_{C^{0,1}})d_y.
		\end{align*}
		This proves claim a.\smallskip

		Let $p:= x+r_0\theta\in \partial D$, $q_r: = ( x'+r\theta', \phi(x'+r\theta') )$, and let $z_r$ be the intersection point of the line going through $p$ and $q_r$ as well as the line going through $0$ and $x$ for all $r\in (0,r_0)$. That is $z_r:= p+\frac{r_0}{r_0-r}(q_r-p)$, see \autoref{fig:geometry_1}.
		
		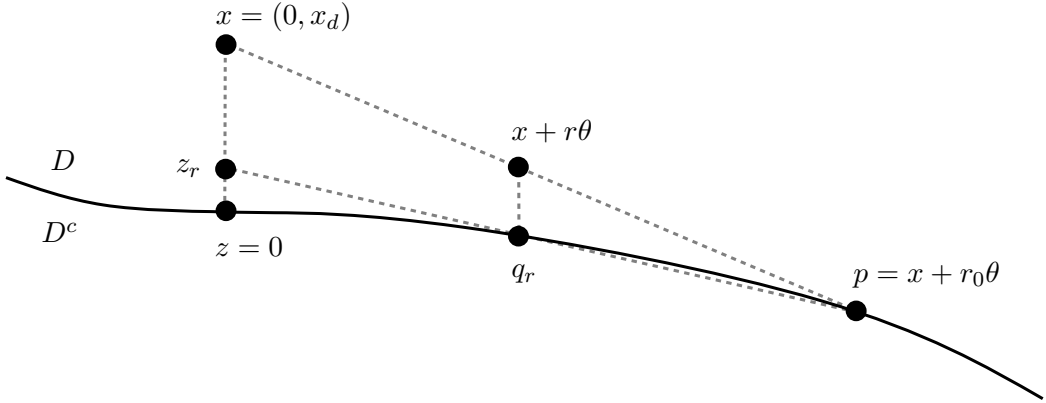
\begin{figure}[!ht]
			\centering
			\begin{tikzpicture}[y=1cm, x=1cm, yscale=1.2,xscale=1.2, inner sep=0pt, outer sep=0pt]
				\begin{scope}[shift={(-0.11101000000000001,6.07435)}]
					\path[draw=gray,dash pattern=on 0.071cm off 0.071cm,even odd
					rule,draw opacity=1,line join=bevel,line width=0.04cm,dash phase=1.938cm]
					(13.7307,12.4223) -- (13.7336,11.6733);

					\path[draw=gray,dash pattern=on 0.071cm off 0.071cm,even odd rule,draw
					opacity=1,line join=bevel,line width=0.04cm,dash phase=1.995cm]
					(17.4410,10.8242) -- (10.4972,12.4261);

					\path[draw=gray,dash pattern=on 0.069cm off 0.069cm,even odd rule,draw
					opacity=1,line join=bevel,line width=0.04cm] (10.5600,13.7784) -- (17.4986,10.8126);

					\path[draw=gray,dash pattern=on 0.061cm off 0.061cm,even odd
					rule,draw opacity=1,line join=bevel,line width=0.04cm,dash phase=1.714cm]
					(10.5093,11.9623) -- (10.5151,13.8787);
					
				\end{scope}
				\path[draw=black,even odd rule,draw opacity=1,line join=bevel,line
				width=0.04cm] (8.0007,18.3867) .. controls (8.3354, 18.2667) and (8.6701,
				18.1467) .. (9.0691, 18.0840) .. controls (9.4681, 18.0213) and (10.0239,
				18.0148) .. (10.2485, 18.0108) .. controls (10.4731, 18.0067) and (10.5903,
				18.0010) .. (10.9387, 17.9986) .. controls (11.2871, 17.9961) and (11.8861,
				17.9959) .. (13.0136, 17.8381) .. controls (14.1411, 17.6802) and (15.7971,
				17.3648) .. (16.9435, 17.0228) .. controls (18.0898, 16.6807) and (18.7264,
				16.3120) .. (19.3630, 15.9433);

				\path[draw=black,fill=black,line join=bevel,line width=0.03cm] (10.4123,
				19.8564) ellipse (0.1021cm and 0.1015cm);

				\path[shift={(-0.11101000000000001,6.07435)},draw=black,fill=black,line
				join=bevel,line width=0.03cm] (10.5230, 11.9453) ellipse (0.1021cm and
				0.1015cm);

				\path[draw=black,fill=black,line join=bevel,line width=0.03cm] (17.3180,
				16.9136) ellipse (0.1021cm and 0.1015cm);

				\path[draw=black,fill=black,line join=bevel,line width=0.03cm] (13.6156,
				18.5033) ellipse (0.1021cm and 0.1015cm);

				\path[draw=black,fill=black,line join=bevel,line width=0.03cm] (13.6170,
				17.7391) ellipse (0.1021cm and 0.1015cm);

				\path[draw=black,fill=black,line join=bevel,line width=0.03cm] (10.4074,
				18.4807) ellipse (0.1021cm and 0.1015cm);

				\path[shift={(-0.11101000000000001,6.07435)},fill=black,even odd rule,line
				join=bevel,line width=0.081cm] (10.4,11.45) node[above right]
				(text1){$z=0$};

				\path[shift={(-0.11101000000000001,6.07435)},fill=black,dash pattern=on 0.061cm
				off 0.061cm,even odd rule,draw opacity=1,line join=bevel,line
				width=0.061cm,dash phase=1.714cm] (8.6,12.4) node[above right]
				(text3){$D$};

				\path[shift={(-0.11101000000000001,6.07435)},fill=black,dash pattern=on 0.061cm
				off 0.061cm,even odd rule,draw opacity=1,line join=bevel,line
				width=0.061cm,dash phase=1.714cm] (8.5,11.6) node[above right]
				(text4){$D^c$};

				\path[shift={(-0.11101000000000001,6.07435)},fill=black,dash pattern=on 0.284cm
				off 0.142cm on 0.071cm off 0.142cm,even odd rule,draw opacity=1,line
				join=bevel,line width=0.071cm] (13.65,12.7) node[above right]
				(text10){$x+r\theta$};

				\path[shift={(-0.11101000000000001,6.07435)},fill=black,dash pattern=on 0.284cm
				off 0.142cm on 0.071cm off 0.142cm,even odd rule,draw opacity=1,line
				join=bevel,line width=0.071cm] (17.4,11.1) node[above right]
				(text11){$p=x+r_0\theta$};

				\path[shift={(-0.11101000000000001,6.07435)},fill=black,dash pattern=on 0.284cm
				off 0.142cm on 0.071cm off 0.142cm,even odd rule,draw opacity=1,line
				join=bevel,line width=0.071cm] (13.65,11.15) node[above right]
				(text12){$q_r$};

				\path[fill=black,dash pattern=on 0.061cm off 0.061cm,even odd rule,draw
				opacity=1,line join=bevel,line width=0.061cm,dash phase=1.714cm]
				(10.3,20) node[above right] (text2){$x=(0,x_d)$};

				\path[fill=black,dash pattern=on 0.071cm off 0.071cm,even odd rule,draw
				opacity=1,line join=bevel,line width=0.071cm] (9.86,18.4) node[above
				right] (text13){$z_r$};
				
			\end{tikzpicture}
			\caption{Geometry close to $\partial D$ in case 1. }
			\label{fig:geometry_1}
		\end{figure}
		
		\textit{Claim b.} There exists a constant $c_2>0$ such that $\abs{z_r-0}\le 1/2\, \abs{0-x}$. \smallskip
		
		Now, we prove claim b. By \eqref{eq:phi_properties} and the monotonicity of $\omega$, we estimate
		\begin{align*}
			\abs{(z_r)_d}&= \big|\frac{r_0\phi(r\theta')- r\cdot \phi(r_0\theta')}{r_0-r}\big|\le \abs{\phi(r\theta')}+|\frac{r}{r_0-r}\int\limits_{r_0}^{r}\theta'\cdot \nabla \phi(t\theta')\d t|\\
			&\le \phi(r\theta')+r \sup_{r<t<r_0}\abs{\nabla \phi(t\theta')}\le r \omega(r)+ r \sup\limits_{r<t<r_0}\omega(t\abs{\theta'})\le 2r_0\omega(r_0)=2 f(r_0).
		\end{align*}
		By $r_0\le f^{-1}(ax_d)$ and the previous bound on $(z_r)_d$, we find 
		\begin{align*}
			\abs{z_r-0}= \abs{(z_r)_d}\le 2 f(f^{-1}(ax_d))\le  \frac{x_d}{2}= \frac{\abs{0-x}}{2}.
		\end{align*}
		Here, we used $a<1/4$. This proves the claim b. \smallskip
		
		\textit{Claim c.} The line segment $\{x+r\theta\mid r\in (0,f^{-1}(ax_d))\}$ hits $\partial D$ at most once.\smallskip
		
		Assume the claim is false. Then by intermediate value theorem, there exists a point $\xi=(t\theta', \phi(t\theta'))\in \partial D \cap \big(B_{f^{-1}(ax_d)}\setminus B_{x_d/2}\big)$ such that the line $\{ x+r\theta\mid r>0\}$ is tangent to the surface $\partial D$ at $\xi$. Thus, $\theta'\cdot\nabla \phi(\xi')= \theta_d$ which implies
		\begin{align*}
			\abs{\theta_d}= \abs{\theta'\cdot\nabla \phi(\xi')}\le \abs{\theta'}\omega(\abs{\xi'})\le \omega(f^{-1}(ax_d)). 
		\end{align*}
		Using this bound on the angle $\theta$, we can bound the distance of the line segment $\{x+r\theta\mid 0<r<f^{-1}(ax_d)\}$ from $\partial D$ as follows
		\begin{equation}\label{eq:claim_c_final_step}
			\begin{split}
				\distw{x+r\theta}{\partial D}&\ge c_1 \abs{x_d+r\theta_d-\phi(r\theta')}\ge c_1 \Big(x_d - r \abs{\theta_d}-\abs{\phi(r\theta')}\Big)\\
				&\ge c_1 \Big(x_d - f^{-1}(a x_d)\omega(f^{-1}(ax_d)) - r \omega(r)\Big)\ge c_1 \big(1- 2a\big)x_d.
			\end{split}
		\end{equation}
		This is a contradiction! Thus, claim c is proven. \medskip
		
		Now, we proceed to estimate the term $\II{2}$ in the case 1. By the intercept theorem, we find
		\begin{align*}
			\frac{\abs{x+r\theta- q_r}}{\abs{z_r-x}}= \frac{\abs{p-(x+r\theta)}}{\abs{p-x}}
		\end{align*}
		which implies, using claim a, claim b, and \eqref{eq:reg_distance_properties},
		\begin{align*}
			\fd(x+r\theta)&\ge C_1^{-1}\distw{x+r\theta}{\partial D}\ge \frac{c_1 }{C_1}\abs{x+r\theta- q_r}= \frac{c_1 }{C_1}\frac{\abs{p-(x+r\theta)}}{\abs{p-x}}\abs{z_r-x}\\
			&\ge \frac{c_1 }{C_1}\frac{(r_0-r)}{r_0}(\abs{0-x}-\abs{z_r-0})\ge\frac{c_1 }{2C_1}\frac{r_0-r}{r_0} x_d.
		\end{align*}
		Due to claim c, the domain of integration in the term $\II{2}$ reduces to $[ax_d,r_0]$. This yields 
		\begin{align*}
			\II{2}=&s\int\limits_{ax_d}^{r_0} \fd(x+r\theta)^{s-1}\frac{\omega(r)}{r^{2s}}\d r\le \omega(r_0)s \big(\frac{2C_1}{c_1}\big)^{1-s}x_d^{s-1}\int\limits_{ax_d}^{r_0} \frac{(1-r/r_0)^{s-1}}{r^{2s}}\d r\\
			&\le s \big(\frac{2C_1}{c_1}\big)^{1-s}x_d^{s-1}\omega(r_0)\Big(2^{1-s}\int\limits_{ax_d}^{r_0/2} \frac{1}{r^{2s}}\d r+\int\limits_{r_0/2}^{r_0} (1-r/r_0)^{s-1}\frac{2^{2s}}{r_0^{2s}}\d r\Big).
		\end{align*}
		Here, we used the monotonicity of $\omega$. In the first integral on the right-hand side of the previous estimate we scale $r$ by $ax_d$, use $r_0\le f^{-1}(ax_d)$, and \eqref{eq:property_of_f}. In the second integral, we simply integrate. This yields
		\begin{align*}
			\II{2}&\le s \big(\frac{2C_1}{c_1}\big)^{1-s}x_d^{s-1}\omega(r_0)\Big(2^{1-s}(ax_d)^{1-2s}\int\limits_{1}^{\frac{f^{-1}(ax_d)}{2ax_d}} \frac{1}{r^{2s}}\d r+r_0^{1-2s}\frac{2^{s}}{s} \Big)\\
			&\le  \frac{4C_1}{c_1}\Big(s a^{1-2s}x_d^{-s}\omega(f^{-1}(ax_d))\int\limits_{1}^{\frac{1}{2\omega(f^{-1}(ax_d))}} \frac{1}{r^{2s}}\d r+x_d^{s-1}\omega(r_0)r_0^{1-2s} \Big).
		\end{align*}
		The first term can be estimated as claimed in the lemma. We estimate the second term just as the term $\II{1,2}$ in the sub-case 2.1. Here, we use  $ax_d\le r_0\le f^{-1}(ax_d)$ and \eqref{eq:property_of_f}. \smallskip
		
		\textbf{Case 2.} The ray $\{x+r\theta\mid r\in (0,f^{-1}(ax_d))\}$ does not hit the boundary $\partial D$.  \smallskip
		
		We make a few more distinctions. \smallskip
		
		\textbf{Sub-case 2.1.} If $\theta_d\ge 0$, then just as in \eqref{eq:claim_c_final_step}
		\begin{equation*}
			\distw{x+r\theta}{\partial D}\ge c_1 \big(x_d+r\theta_d- \abs{\phi(r\theta')}\big)\ge c_1 \big(1-a\big)x_d.
		\end{equation*}
		This yields
		\begin{align*}
			\II{2}&\le c_1^{s-1}C_1^{1-s}(1-a)^{s-1}\int\limits_{ax_d}^{f^{-1}(ax_d)}x_d^{s-1}\frac{\omega(r)}{r^{2s}}\d r \le  \frac{4C_1}{3c_1}  a^{1-2s} \frac{\omega(f^{-1}(ax_d))}{x_d^{s}} \int\limits_{1}^{\frac{1}{\omega(f^{-1}(ax_d))}} \frac{\d r}{r^{2s}}
		\end{align*}
		which implies the desired estimate as in the bound for the term $\II{1}$ in the case 1. Here, we used \eqref{eq:property_of_f}. \smallskip
		
		\textbf{Sub-case 2.2.} We assume $\theta_d<0$. In this case, where $\tilde{\theta}$ is chosen such that $\tilde{\theta}'$ and $\theta'$ are linear dependent and $\tilde{\theta}_d<\theta_d$ such that $x+f^{-1}(ax_d)\tilde{\theta}\in \partial D$, see \autoref{fig:subcase23}. 
		
		\begin{figure}[!ht]
			\centering
			\begin{tikzpicture}[y=1cm, x=1cm, yscale=1,xscale=1, inner sep=0pt, outer sep=0pt]
				
				\draw[draw=gray,dash pattern=on 0.061cm off 0.061cm,even odd
				rule,draw opacity=1,line join=bevel,line width=0.04cm,dash phase=1.714cm] (10.4,18.5) -- (17.7,16) node[color=black, pos=0.5, above, rotate=-17, xshift=2, yshift=3] (text1){$f^{-1}(ax_d)$};
				
				\draw[draw=gray,dash pattern=on 0.061cm off 0.061cm,even odd
				rule,draw opacity=1,line join=bevel,line width=0.04cm,dash phase=1.714cm] (10.4,18.5) -- (17.2,14.95);
				
				\draw[draw=gray,dash pattern=on 0.061cm off 0.061cm,even odd
				rule,draw opacity=1,line join=bevel,line width=0.04cm,dash phase=1.714cm] (10.4123,17.8519) +(-26.5:7.48) arc (-26.5:7:7.48);
				
				\draw[draw=gray,dash pattern=on 0.061cm off 0.061cm,even odd
				rule,draw opacity=1,line join=bevel,line width=0.04cm,dash phase=1.714cm] (10.4,18.5) -- (10.4,
				16.0152);
				
				\path[shift={(0,-2.0045200000000003)},draw=black,even odd rule,draw
				opacity=0.993,line join=bevel,line width=0.071cm] (8.0007,18.3867) .. controls
				(8.3354, 18.2667) and (8.6701, 18.1467) .. (9.0691, 18.0840) .. controls
				(9.4681, 18.0213) and (10.0239, 18.0148) .. (10.2485, 18.0108) .. controls
				(10.4731, 18.0067) and (10.5903, 18.0010) .. (10.9387, 17.9986) .. controls
				(11.2871, 17.9961) and (11.8861, 17.9959) .. (13.0163, 17.8148) .. controls
				(14.1464, 17.6337) and (15.8077, 17.2717) .. (16.7729, 17.0338) .. controls
				(17.7380, 16.7960) and (18.0070, 16.6822) .. (18.2760, 16.5684);

				\path[draw=black,fill=black,line join=bevel,line width=0.03cm] (10.4,18.5) ellipse (0.1021cm and 0.1015cm) node[above, xshift=0.8cm, yshift=0.1cm] (text1){$x=(0,x_d)$};

				\path[draw=black,fill=black,line join=bevel,line width=0.03cm] (10.4,
				16.0152) ellipse (0.1021cm and 0.1015cm) node[below, xshift=10, yshift=-5] (text1){$z=0$};

				\path[draw=black,fill=black,line join=bevel,line width=0.03cm] (17.29,
				14.90) ellipse (0.1021cm and 0.1015cm) node[above right, xshift=0.3cm] (text1){$x+f^{-1}(ax_d)\tilde{\theta}$};

				\path[draw=black,fill=black,line join=bevel,line width=0.03cm] (17.7,16) ellipse (0.1021cm and 0.1015cm) node[above right, xshift=0.3cm] (text1){$x+f^{-1}(ax_d)\theta$};

				\path[fill=black,dash pattern=on 0.061cm off 0.061cm,even odd rule,draw
				opacity=0.993,line join=bevel,line width=0.061cm,dash phase=1.714cm]
				(8.5,16.5) node[above right] (text3){$D$};

				\path[fill=black,dash pattern=on 0.061cm off 0.061cm,even odd rule,draw
				opacity=0.993,line join=bevel,line width=0.061cm,dash phase=1.714cm]
				(8.4,15.5) node[above right] (text4){$D^c$};
				
			\end{tikzpicture}
			\caption{Geometry near $\partial D$ in sub-case 2.2.}\label{fig:subcase23}
		\end{figure}
		
		\textit{Claim d.} The vertical distance of $x+r\theta$ to the boundary $\partial D$ is bounded from below by a multiple of the vertical distance of $x+ r \tilde{\theta}$ to $\partial D$ for $0<r<f^{-1}(ax_d)$. \smallskip
		
		We choose $c:= (1-a)/(1+a)$, then 
		\begin{equation}\label{eq:claim_d_final_step}
			\begin{split}
				(1-c)x_d&\ge (1+c)ax_d= (1+c)f(f^{-1}(ax_d))\ge f(r\abs{\theta'})+ cf(r\abs{\tilde{\theta}'})\\
				&\ge \abs{\phi(r\theta')}+c\abs{\phi(r\tilde{\theta}')}\ge \phi(r\theta')-c\phi(r\tilde{\theta}')\ge \phi(r\theta')-c\phi(r\tilde{\theta}')-r(\theta_d-c\tilde{\theta}_d).
			\end{split}
		\end{equation}
		Here, we used $r<f^{-1}(ax_d)$, the monotonicity of $f$, $\abs{\theta'},|\tilde{\theta}'|< 1$, $c<1$, and $\theta_d\ge \tilde{\theta}_d$. Due to claim c, both line segments $\{x+r\theta\mid r\in [0,f^{-1}(ax_d))\}$ and $\{x+r\tilde{\theta}\mid r\in [0,f^{-1}(ax_d))\}$ are in $D$. Thus, we know that $x_d+r\theta_d-\phi(r\theta')$ and $x_d+r\tilde{\theta}_d-\phi(r\tilde{\theta}')$ are positive. By \eqref{eq:claim_d_final_step}, we find
		\begin{equation*}
			\abs{x_d+r\theta_d-\phi(r\theta')}\ge c\abs{x_d+r\tilde{\theta}_d-\phi(r\tilde{\theta}')}.
		\end{equation*}
		This proves claim d. \medskip		
		
		 Now, we are in the case 1 and, thus, the result follows by 
		 \begin{align*}
		 	\II{2}&\le sc_1^{s-1}\int\limits_{ax_d}^{f^{-1}(ax_d)} \abs{x_d+r\theta_d- \phi(x'+r\theta')}^{s-1}\frac{\omega(r)}{r^{2s}}\d r\\
		 	&\le s(c_1c)^{s-1}\int\limits_{ax_d}^{f^{-1}(ax_d)} \abs{x_d+r\tilde{\theta}_d- \phi(x'+r\tilde{\theta}')}^{s-1}\frac{\omega(r)}{r^{2s}}\d r
		 \end{align*}
	 	and the bound on $\II{2}$ in case 1. 
	\end{proof}

	\begin{remark}\label{rem:sharp_lemma}
		In particular, the bound in \autoref{lem:intermediate_estimate_1} is sharp which can be seen by considering a $\cdini$-parabola with apex in the origin for $D$ and choosing a horizontal direction for $\theta$.
	\end{remark}
	
	We will construct barriers using a combination of the function $\fd^s$ and $\fd^s\, \zeta(\fd)$. In the next two propositions we study the behavior of these functions under $A_s$. We start with the function $\fd^s$. The main observation is that $x\mapsto (x_0+b\cdot x)_+^s$ is harmonic under the operator $A_s$ in the region where it does not vanish.
	
	\begin{proposition}\label{prop:almost_harmonic}
		Let $s_0\in (0,1)$ and $D,\omega, \iota$ be as in \autoref{lem:intermediate_estimate_1}. There exists a constant $C=C(d,s_0,\omega, \Lambda)>0$ such that for any $s\in (s_0,1)$
		\begin{equation}\label{eq:almost_harmonic}
			\abs{A_s(\fd^s)(x)}\le C+C\frac{\omega(d_D(x))}{d_D(x)^{s}}+  C \frac{\omega(f^{-1}(d_D(x)))}{d_D(x)^s} \begin{cases}
				\frac{(\omega(f^{-1}(d_D(x))))^{2s-1}- 1}{1-2s}&, s\ne 1/2 \\
				\ln\big( \frac{1}{\omega(f^{-1}(d_D(x)))} \big)&, s=1/2
			\end{cases}, 
		\end{equation}
		 for $x\in D\cap B_1$. Here, the function $f$ is given by $f(t) = t \omega(t)$. If, additionally, the L{\'e}vy-measure $\nu_s$ satisfies the assumption \eqref{ass:pUB}, then 
		\begin{equation}\label{eq:almost_harmonic_with_pub}
			\abs{A_s(\fd^s)(x)}\le C+C\frac{\omega(d_D(x))}{d_D(x)^{s}}\quad\text{ for }x\in D\cap B_1.
		\end{equation}
	\end{proposition}
	\begin{proof}
		We fix $x\in D$ and set $\eps:= d_D(x)=\distw{x}{D^c}$. Let $\rho_1$ be the constant from \autoref{lem:intermediate_estimate_1}. If $\eps\ge \rho/2$, then $\abs{A_s\fd^s}\le C$ by the smoothness of the regularized distance function $\fd$. So, let's assume that $\distw{x}{\partial D}=\eps< \rho/2$. We define the auxiliary function $l$ by 
		\begin{equation}\label{eq:aux_l}
			l(y)=\big(\fd(x)+\nabla \fd(x)\cdot (y-x)\big)_+.
		\end{equation}
		Note that the $s$-th power of this function solves $A_s l^s(y)=0$ for any $y\in \R^d$ such that $l(y)>0$, see \cite[Lemma 2.6.2]{RoFe24}. This is easily seen as follows:
		\begin{align*}
			A_s l^s(x+y)&= (1-s)\int\limits_{S^{d-1}} \pv\int\limits_{\R} \frac{\fd(x)^s- (\fd(x)+\nabla \fd(x)\cdot \theta r)_+^s}{\abs{r}^{1+2s}}\d r \mu(\d \theta)\\
			&= \int\limits_{S^{d-1}} \abs{\nabla\fd(x)\cdot \theta }(-\Delta)_\R^s[  (\cdot)_+^s ](\frac{\fd(x)}{\abs{\nabla\fd(x)\cdot \theta }}) \mu(\d \theta)= 0.
		\end{align*}
		In the last equality, we used that $(\cdot)_+^s$ is an $s$-harmonic function on the half line, see e.g.\ \cite[Theorem 2.4.1]{BuVa16}.
		
		\smallskip
		
		Depending on the size of $y$, we will derive estimates of $\abs{\fd^s_{x+y}-l^s(x+y)}$. For this, we define 
		\begin{equation*}
			a:= (\norm{\fd}_{C^1}(1+2C_1)+4)^{-1}.
		\end{equation*}
		\medskip
		
		\textbf{Case $\abs{y}<a\eps$.} Since $\fd(x)=l(x)$ and $\nabla \fd(x)=\nabla l(x)$, we may make use of the higher regularity of both $\delta$ and $l$ away from the boundary. We calculate
		\begin{align*}
			\fd(x+y)&\ge C_1^{-1}\distw{x+y}{\partial D}\ge C_1^{-1}\big( \distw{x}{\partial D} - \abs{y}\big)\ge \frac{1-a}{C_1}\eps\\
			\intertext{and}
			l(x+y)&\ge \fd(x)+\nabla \fd(x)\cdot y\ge C_1^{-1}\eps- \abs{\nabla \fd(x)}a\eps\ge \frac{\eps}{2C_1}.
		\end{align*}
		Using this, \eqref{eq:reg_distance_properties}, and the concavity of $(\cdot)_+^s$, we derive the following bound
		\begin{align*}
			\abs{\fd(x+y)^s-l^s(x+y)}&\le s\abs{\fd(x+y)-l(x+y)}(\fd(x+y)^{s-1}+l^{s-1}(x+y))\\
			&\le s2 (2C_1)^{1-s} \eps^{s-1}\abs{\fd(x+y)-l(x+y)}\\
			&\le s2 (2C_1+1) \eps^{s-1} \norm{\fd}_{C^2(B_{a\eps}(x))}\abs{y}^2\le sc_3 \omega(\eps)\eps^{s-2}\abs{y}^{2}.
		\end{align*}
	
		\textbf{Case $\abs{y}\in[a\eps,\rho_1)$.} Using Taylor's theorem, there exists a constant $c_1>0$ such that
		\begin{equation}
			\abs{\fd(x+y)-l(x+y)}\le c_1\omega(\abs{y})\abs{y}\quad \text{ for all }y\in \R^d,
		\end{equation}
		see also \cite[Lemma B.2.3]{RoFe24}. Thus, for any $x+y\in D$ using the concavity of $(\cdot)_+^s$
		\begin{align*}
			\abs{\fd(x+y)^s-l^s(x+y)}&\le s(\fd(x+y)^{s-1} + l(x+y)^{s-1})\abs{\fd(x+y)-l(x+y)}\\
			&\le sc_1(\fd(x+y)^{s-1} + l(x+y)^{s-1}) \abs{y} \omega(\abs{y}).
		\end{align*}
	
		\textbf{Case $\abs{y}\ge \rho_1$.} Here, we use the trivial bounds $\delta_{x+y}^s\le c_2 \abs{y}^s$ and $l(x+y)^s\le c_2 \abs{y}^{s}$.\medskip
				
		We estimate $A_s(\fd^s)(x)$ by splitting it into three parts.
		\begin{align*}
			\abs{A_s(\fd^s)(x)}&= \abs{A_s(\fd^s-l^s)(x)}\le  \I{}+\II{}+\III{}\\
			&:= \Bigg(\,\int\limits_{B_{a\eps}}+ \int\limits_{B_{\rho_1}\setminus B_{a\eps}} + \int\limits_{B_{\rho_1}^c}\Bigg) \abs{\fd^s(x+y)-l^s(x+y)}\nu_s(\d y). 
		\end{align*}
		Using the previous bound derived in the case $\abs{y}<a\eps$, we may estimate $\I{}$ by 
		\begin{align*}
			sc_3\omega(\eps)\eps^{s-2} \int\limits_{B_{a\eps}} \abs{y}^2\nu_s(\d y)= 2s(1-s)\omega(\eps)\eps^{s-2}\int\limits_{S^{d-1}}\int\limits_{0}^{a\eps} r^{1-2s}\d r \mu(\d \theta)\le 2sa^{2-2s}\Lambda\frac{\omega(\eps)}{\eps^s}.
		\end{align*}
		This is the desired bound. We continue by estimating the term $\III{}$ using the previous trivial growth bound of $\fd^s$ and $l^s$. This leads to
		\begin{align*}
			\III{}&\le 2(1-s)\Lambda c_2\int\limits_{\rho_1}^\infty r^{-1-s}\d r \le  \frac{2(1-s)\Lambda c_2}{s}\rho_1^{-s}.
		\end{align*}
	
		Now, we turn our attention to the term $\II{}$. Upon analyzing the term $\II{}$ with more care, one notices this term behaves differently depending on whether \eqref{ass:pUB} holds or not. Thus, we make a case distinction. 
		
		\textbf{Case (}\eqref{ass:pUB} holds\textbf{).} We assume that \eqref{ass:pUB} holds. Since $r\mapsto \omega(r)/r^{s/2}$ is decreasing in a neighborhood of $0$, we apply \cite[Lemma B.2.4]{RoFe24} to find that
		\begin{align*}
			\II{}&\le s(1-s)\Lambda c_1 \frac{\omega(a\eps)}{(a\eps)^{s/2}}\int\limits_{D \cap B_{\rho_1}\setminus B_{a\eps}}(\delta_{x+y}^{s-1} + l(x+y)^{s-1}) \abs{y}^{1-d-3s/2}\d y\le (1-s)c_3(1+ \frac{\omega(\eps)}{\eps^s}).
		\end{align*}
		This is the desired bound. \smallskip
		
		\textbf{Case (}\eqref{ass:pUB} does not hold\textbf{).} Here, the bound follows directly from \autoref{lem:intermediate_estimate_1}. 
	\end{proof}

	As announced above, in the following proposition we study the behavior of $\fd^s\zeta(\fd)$ under appropriate assumptions on the function $\zeta$, see \autoref{sec:one_dimensional}. This will be the last tool needed to construct the barriers. 

	\begin{proposition}\label{prop:subsolution_d}
		Let $s_0\in (0,1)$ and $D,\omega, \iota$ be as in \autoref{lem:intermediate_estimate_1}. Suppose the function $\zeta$ satisfies \ref{ass:zeta_pic}, \ref{ass:zeta_3rd}, \ref{ass:zeta_ndecr}, \ref{ass:zeta_iota}, and \ref{ass:zeta_growth}. There exist two positive constants $C_4$ and $C_5$ depending only on $s_0$, $\omega$, $\zeta$, and the constants from \autoref{prop:subsolution_d_1} such that
		\begin{equation}\label{eq:barrier_main_estimate_without_pUB}
			\begin{split}
				A_s[\fd^s\zeta(\fd)](x)&\le - C_{4} d_D(x)^{1-s}\zeta^\prime(d_D(x))+C_5+ C_5 \frac{\omega(d_D(x))}{d_D(x)^s}\\
				&\qquad+ C_5 \frac{\omega(f^{-1}(d_D(x)))}{d_D(x)^s} \begin{cases}
					\frac{(\omega(f^{-1}(d_D(x))))^{2s-1}- 1}{1-2s}&, s\ne 1/2 \\
					\ln\big( \frac{1}{\omega(f^{-1}(d_D(x)))} \big)&, s=1/2
				\end{cases}
			\end{split}
		\end{equation}
		for all $x\in D\cap B_1$. If additionally \eqref{ass:pUB} holds, then 
		 \begin{equation}\label{eq:barrier_main_estimate_with_pUB}
		 	A_s[\fd^s\zeta(\fd)]\le  - C_{4} d_D(x)^{1-s}\zeta^\prime(d_D(x))+C_5+C_{5} \frac{\omega(d_D(x))}{d_D(x)^s}
		 \end{equation}
		for all $x\in D\cap B_1$. 
	\end{proposition}
	\begin{proof}
		Again, we define the auxiliary function $l$ as in \eqref{eq:aux_l}. Fix $x\in D$ and $\eps:= \distw{x}{\partial D}$. If $\eps\ge \rho/2$, then the estimate is trivial due to the regularity of $\fd^s\zeta(\fd)$. So we assume the contrary. For any $y\in \R^d$ such that $l(y)>0$, by \autoref{prop:subsolution_d_1} there exists a positive constant $c_1$ depending only on $\omega, s_0,D$ such that
		\begin{equation*}
			A_s[l^s\zeta(l)](x)\le - \Lambda c_1  \eps^{1-s} \zeta^\prime(\eps).
		\end{equation*}
		As in the proof of \autoref{prop:almost_harmonic}, it remains to prove an appropriate bound on 
		\begin{align*}
			\I{}:=\int\limits_{\R^d} \abs{\fd^s\zeta(\fd)-l^s\zeta(l)}(x+y)\nu_s(\d y).
		\end{align*}
		Since $\zeta$ is bounded and concave, the following claim holds true.\smallskip
		
		\textbf{Claim 1.} $\abs{t^s\zeta(t)-r^s\zeta(r)}\le 2\frac{s+1}{s}\norm{\zeta}_{L^\infty(0,t_1)}\abs{t^s-r^s}$ for all $0<r<t<t_1$. \smallskip
		
		We prove this claim. Without loss of generality we assume $t\ge r$, say $t=r+\eps$. Now, it is left to show for any $r\ge 0$, $\eps\ge 0$
		\begin{equation*}
			(r+\eps)^s\zeta(r+\eps)-r\zeta(r)\le 2\norm{\zeta}_{L^\infty(0,t_1)}((r+\eps)^s-r^s).
		\end{equation*}
		This inequality is true for $\eps =0$. We differentiate with respect to $\eps$. Then it remains to prove 
		\begin{align*}
			s\zeta(r+\eps)+ (r+\eps)\zeta'(r+\eps)\le s2\frac{s+1}{s}\norm{\zeta}_{L^\infty(0,t_1)} .
		\end{align*}
		This is true since $\zeta$ is concave and, thus, $t\zeta'(t)\le \zeta(t)$. \medskip
		
		Now, we fix a very large constant $t_1$. We will split the integration domain of $\I{}$ into $B_1(0)$ and its complement. If $y\in B_1(0)$, then 
		\begin{align*}
			\fd(x+y)&\le C_1 \distw{x+y}{D^c}\le C_1 (\rho/2\wedge t_0 \,+1 ),\\
			l(x+y)&\le \fd(x)+\abs{\nabla \fd(x)\cdot y }\le C_1 \rho/2\wedge t_0+\norm{\fd}_{C^1}. 
		\end{align*}
		Here, we used \eqref{eq:reg_distance_properties}. If we set $c_2:= C_1 (\rho/2\wedge t_0 \,+1 )+\norm{\fd}_{C^1}$, then, using the claim 1, we find 
		\begin{align*}
			\int\limits_{B_1(0)} \abs{\fd^s\zeta(\fd)- l^s \fd(l)}(x+y)\nu_s(\d y)&\le 2\frac{s+1}{s}\zeta(c_2) \int\limits_{\R^d} \abs{\fd^s-l^s}(x+y)\nu_s(\d y)
		\end{align*}
		Now, an application of \autoref{prop:almost_harmonic} yields the desired bound. Outside of the ball $B_1(0)$, we use the trivial bound $\abs{\fd(x+y)-l(x+y)}\le c_3 (1+\abs{y})$. This proves the result just as in the estimate of the term $\III{}$ in the proof of \autoref{prop:almost_harmonic}. 		
	\end{proof}
	Next, we prove the existence of a function $\zeta$ as required in \autoref{prop:subsolution_d}. We prove two results depending on whether \eqref{ass:pUB} holds or not.
	\begin{lemma}\label{lem:zeta_with_pUB}
		Let $0<\iota<\min\{s,1/3\}$, $\omega$ be the modulus of continuity modified as in \autoref{rem:replace_modulus}, and $C>0$. Then there exists $t_0>0$ and $\zeta:[0,\infty)\to [0,\infty)$ continuous on $[0,\infty)$ and in $C^{2}((0,\infty))\cap C^{3}((0,2))$ which satisfies \ref{ass:zeta_pic}, \ref{ass:zeta_3rd}, \ref{ass:zeta_ndecr}, \ref{ass:zeta_iota}, and \ref{ass:zeta_growth}. Furthermore, $\zeta$ satisfies 
		\begin{equation}\label{eq:zeta_omega_coupling_pUB}
				C\big(t^s+\omega(t)\big)\le t \zeta^\prime(t).
		\end{equation}
		for all $0<t<t_0$. 
	\end{lemma}
	\begin{proof}
		\textit{Step 1.} In the first step, we define the function $\zeta$ in a neighborhood of $0$ such that \ref{ass:zeta_pic}-\ref{ass:zeta_iota} are satisfies. \smallskip
		
		We define the function $\zeta:[0,\infty)\to (0,\infty)$ via
		\begin{equation}\label{eq:def_zeta_with_pUB}
			\zeta(t):= C\int\limits_{0}^t \frac{r^{\iota}+\omega(r)}{r}\d r.
		\end{equation} 
		Since $\omega\in C^2$ and satisfies \eqref{eq:dini}, the function $\zeta$ is in $C([0,\infty))\cap C^3((0,\infty))$. Now, we check all assumptions \ref{ass:zeta_pic}-\ref{ass:zeta_iota}.\smallskip
		
		Since $\omega$ satisfies \eqref{eq:dini}, $\zeta$ is continuous and $\zeta(0)=0$. Additionally, $\zeta$ is positive in $(0,\infty)$ since $\omega$ is nonnegative. Its derivative 
		\begin{equation}\label{eq:zeta_prime_calc}
			\zeta'(t)= C \frac{t^{\iota}+\omega(t)}{t}
		\end{equation}
		is positive and, thus, $\zeta$ is increasing. Its second derivative equals
		\begin{equation}\label{eq:zeta_prime_prime_calc}
			\zeta^{\prime\prime}(t)= -C (1-\iota)t^{\iota-2}- C \frac{\omega(t)-t\omega'(t)}{t^2}
		\end{equation}
		which is negative since $\omega$ is concave. Thus, the function $\zeta$ is concave and \ref{ass:zeta_pic} holds true. \smallskip
		
		A minor calculation reveals 
		\begin{align*}
			t^2 \zeta^{\prime\prime\prime}(t) = (2-\iota)(1-\iota)Ct^{\iota-1}+C\frac{2\omega(t)-2t\omega^\prime(t)+t^2\omega^{\prime\prime}(t)}{t}.
		\end{align*} 
		Now using \eqref{eq:modulus_second_derivative} and the previous calculations, the inequality \ref{ass:zeta_3rd} is satisfied if 
		\begin{align*}
			\omega(t)-2t\omega^\prime(t)-3t\omega^{\prime}(t)\ge -c_1 \omega(t).
		\end{align*}
		Since $\omega$ is concave, the assumption \ref{ass:zeta_3rd} is satisfied with the constant $c_1=4$. \smallskip
		
		We calculate the derivative of $t\zeta'(t)$. It equals
		\begin{align*}
			C \iota t^{\iota-1}+ C \omega^{\prime}(t)
		\end{align*}
		which is trivially positive. Thus, the assumption \ref{ass:zeta_ndecr} holds true for our choice of $\zeta$. \smallskip
		
		Finally, the derivative of $\zeta(t)/t^{\iota}$ reads
		\begin{align*}
			C\frac{t^{\iota}+ \omega(t)}{t^{\iota+1}}-\iota \frac{\zeta(t)}{t^{\iota}+1}\ge C\frac{t^{\iota}+ \omega(t)}{t^{\iota+1}}-\iota \frac{C}{t^{\iota}}\int\limits_{0}^t \frac{r^{\iota}+ \frac{1}{\iota}r\omega^\prime(r)}{r}\d r\ge 0.
		\end{align*}
		Here, we used \eqref{eq:modulus_decreasing_iota}. Thus, \ref{ass:zeta_iota} follows. \smallskip
		
		Note that $\zeta$ as defined in step 1 may not have the correct growth bound \ref{ass:zeta_growth}. In step 2, we will change $\zeta$ away from the origin. But first, we need the following claim. \smallskip
		
		\textit{Claim a.} The function $\zeta$ from step 1 satisfies 
		\begin{equation}\label{eq:zeta_step1_additional_property}
			\zeta'(t)\ge -t\zeta^{\prime\prime}(t).
		\end{equation}
	
		The claim a is an easy consequence from the previous calculations \eqref{eq:zeta_prime_calc} and \eqref{eq:zeta_prime_prime_calc}.\medskip
		
		\textit{Step 2.} Now, we restrict the function $\zeta$ defined in step 1 to $[0,1]$ and redefine it on  $(1,\infty)$ such that \ref{ass:zeta_pic}, \ref{ass:zeta_ndecr}, and \ref{ass:zeta_growth} remain true. This is achieved by choosing 
		\begin{align*}
			\zeta(t)&:= a (t-1+b^{\frac{2}{2-s}})^{s/2}+(\zeta(1)-a\,b^{\frac{s}{2-s}})\text{ for all } t\in(1,\infty) \\
			\intertext{where}
			b&= \big(\frac{2-s}{2} \frac{\zeta'(1)}{-\zeta^{\prime\prime}(1)}\big)^{\frac{2-s}{6-2s}},\quad
			a:=\frac{2}{s b} \, \zeta'(1).
		\end{align*}
		Note that $a,b,c$ are both nonnegative. With this choice $\zeta$ is positive in $(0,\infty)$ and satisfies \ref{ass:zeta_growth} by construction. Note that
		\begin{align*}
			\lim\limits_{t\to 1+} \zeta(t)&= a (0+b^{\frac{2}{2-s}})^{s/2}+(\zeta(1)-a\,b^{\frac{s}{2-s}})= \zeta(1),\\
			\lim\limits_{t\to 1+} \zeta'(t)&= \frac{s}{2}a (0+b^{\frac{2}{2-s}})^{s/2-1}= \frac{s}{2}a b^{-1}= \zeta'(1),\\
			\lim\limits_{t\to 1+} \zeta^{\prime\prime}(t)&= -\frac{s}{2}\frac{2-s}{2}a (0+b^{\frac{2}{2-s}})^{s/2-2}=-\frac{2-s}{2} \zeta'(1)b^{\frac{2s-6}{2-s}}=\zeta^{\prime\prime}(1).
		\end{align*}
		Thus, the function $\zeta$ is in $C^2_{\loc}(\R_+)$. Due to the positivity of $a$ and $b$ the function $\zeta$ is globally increasing and concave. Lastly, we check that \ref{ass:zeta_ndecr} is true. We differentiate the function for $t>1$
		\begin{align*}
			t\zeta'(t)&= \frac{ast}{2} (t-1+b^{\frac{2}{2-s}})^{s/2-1},\\
			\frac{2}{as}\partial_t t\zeta'(t)&=  (t-1+b^{\frac{2}{2-s}})^{s/2-1}- (1-s/2)t (t-1+b^{\frac{2}{2-s}})^{s/2-2}.
		\end{align*}
		This is nonnegative for $t>1$ if and only if the following term is positive.
		\begin{equation}\label{eq:extension_zeta_to_R_plus}
			\begin{split}
				(t-1+b^{2/(2-s)})-(1-s/2)t&\ge -(1-s/2)+ \big(\frac{2-s}{2} \frac{\zeta'(1)}{-\zeta^{\prime\prime}(1)}\big)^{\frac{2}{6-2s}} \\
				&\ge -(1-s/2)+ \big(\frac{2-s}{2}\big)^{\frac{2}{6-2s}}  \ge 0. 
			\end{split}
		\end{equation}
		Here, we used claim a.
	\end{proof}
	
		\begin{lemma}\label{lem:zeta_without_pUB}
		Let $s_0\in (0,1)$, $0<\iota<\min\{s,1/3\}$, $\omega$ be the modulus of continuity modified as in \autoref{rem:replace_modulus}, and $C>0$. Then there exists $t_0>0$ such that for all $s_0\le s<1$ there exists a function $\zeta:[0,\infty)\to [0,\infty)$ continuous on $[0,\infty)$ and in $C^{2}((0,\infty))\cap C^{3}((0,2))$ which satisfies \ref{ass:zeta_pic}, \ref{ass:zeta_3rd}, \ref{ass:zeta_ndecr}, \ref{ass:zeta_iota}, and \ref{ass:zeta_growth}. Furthermore, $\zeta$ satisfies 
		\begin{equation}\label{eq:zeta_omega_coupling_without_pUB}
			C\Big(t^s+\omega(t)+ \frac{\omega(f^{-1}(t))^{2s}- \omega(f^{-1}(t))}{1-2s}\Big)\le t \zeta^\prime(t)
		\end{equation}
		for all $0<t<t_0$. Here, $f(t)= t\omega(t)$. In the case $s=1/2$ the assumption \eqref{eq:zeta_omega_coupling_without_pUB} reads
		\begin{align*}
			C\Big(t^s+\omega(t)+ \omega(f^{-1}(t))\ln\big(\frac{1}{\omega(f^{-1}(t))}\big)\Big)\le t \zeta^\prime(t)
		\end{align*}
		for all $0<t<t_0$.
	\end{lemma}
	\begin{proof}
		\textit{Step 1.} In the first step, we define the function $\zeta$ in a neighborhood of $0$ such that \ref{ass:zeta_pic}-\ref{ass:zeta_iota} hold. \smallskip
		
		We split the definition of $\zeta$ into two parts. $\zeta= \zeta_1+\zeta_2$. Now, we define, just as in \autoref{lem:zeta_with_pUB}, \begin{equation*}
			\zeta_1(t):= C \int\limits_{0}^t \frac{r^\iota+\omega(r)}{r}\d r 
		\end{equation*} 
		for all $t\ge 0$. By \autoref{lem:zeta_with_pUB}, more precisely its proof, this function satisfies all the desired properties.\smallskip
	
		To define $\zeta_2$, we make a few observations. We fix $t_1>0$ such that $\omega(f^{-1})(t_1)\le  (2s)^{1/(1-2s)}$. In the interval $0\le r<2s^{1/(1-2s)}$ the function $r\mapsto (r^{2s}-r)/(1-2s)$ is increasing. Note that $h$ is increasing. Since 
		\begin{equation*}
			\frac{\omega(f^{-1}(t))}{t}= \frac{1}{f^{-1}(t)}\to 0\text{ as }t\to \infty,
		\end{equation*}
		the function $h$ is sublinear. Now, let $\gamma$ the function obtained from \autoref{lem:modulus} applied to the function $h:t\mapsto \omega(f^{-1}(t))$. Let $t_1>0$ be such that $\gamma(t_1)\le 1/e\le 1$ and
		\begin{equation}\label{eq:t1_choice}
			\gamma(t_1)\le a_s:=\begin{cases}
				\min\{ (1-\iota)^{1/(1-2s)}, (2s)^{1/(1-2s)} \} &, s>1/2,\\
				\min\{ \big(\frac{1-2s\iota}{1-\iota}\big)^{1/(1-2s)}, (2s)^{1/(1-2s)}\} &, s< 1/2.
			\end{cases}				
		\end{equation}
		Note that $a_s$ is bounded from below by a positive constant depending only on $\iota$ and $s_0$. \smallskip
		
		We define for any $0\le t\le t_1$
		\begin{equation*}
			\zeta_2(t):= \begin{cases}
				C \int\limits_{0}^t \frac{\gamma(r)^{2s}-\gamma(r)}{(1-2s)r}\d r &, s\ne 1/2,\\
				C \int\limits_{0}^t \frac{\gamma(r)}{r}\ln\big(\frac{1}{\gamma(r)}\big)\d r &, s=1/2.
			\end{cases} 
		\end{equation*} 
		Now, we check the assumptions \ref{ass:zeta_pic}-\ref{ass:zeta_iota} in the interval $[0,t_1]$. Thereafter, we will extend $\zeta_2$ to $\R_+$ appropriately. \smallskip
		
		In the regime $0<t\le t_1$, $\zeta_2$ is positive since $0<\gamma(t)\le 1/e<1$. Its derivative equals
		\begin{align*}
			\zeta_2'(t)= \begin{cases}
				C\frac{\gamma(t)^{2s}-\gamma(t)}{(1-2s)t} &, s\ne 1/2, \\
				C \frac{\gamma(t)}{t}\ln\big(\frac{1}{\gamma(t)}\big)&, s=1/2.
			\end{cases}
		\end{align*}
		Again, by the same reasons as before, this is positive in $(0,t_1)$. To check the concavity, we calculate the second derivative. 
		\begin{align*}
			\zeta_2^{\prime\prime}(t)=\frac{C}{1-2s} \frac{\big(2s\gamma(t)^{2s-1}-1\big)t\gamma'(t) - \gamma(t)^{2s}+ \gamma(t)}{t^2}.
		\end{align*}
		If $s>1/2$, then we estimate this from above by 
		\begin{align*}
			\frac{C}{2s-1} \frac{\gamma(t)^{2s}-(1-\iota) \gamma(t)}{t^2}\le 0
		\end{align*}
		Here, we used \eqref{eq:modulus_decreasing_iota} for $\gamma$ and $\gamma(t)\le \gamma(t_1)\le (1-\iota)^{1/(1-2s)}$. If $s<1/2$, then note that $(2s\gamma(t)^{2s-1}-1)\ge 0$ due to $\gamma(t)\le \gamma(t_1)\le (2s)^{1/(1-2s)}$. Thus, we estimate the second derivative of $\zeta_2$ by
		\begin{align*}
			\zeta_2^{\prime\prime}(t)\le \frac{C}{1-2s} \frac{ - (1-2s\iota)\gamma(t)^{2s}+ (1-\iota)\gamma(t)}{t^2}\le 0.
		\end{align*}
		In the last inequality, we used 
		\begin{equation*}
			\gamma(t)\le \gamma(t_0)\le \big(\frac{1-2s\iota}{1-\iota}\big)^{1/(1-2s)}.
		\end{equation*}
		Lastly, if $s=1/2$, then 
		\begin{align*}
			\zeta_2^{\prime\prime}(t)=C \frac{-t+\big(t\gamma'(t)-\gamma(t)\big)\ln(1/\gamma(t))}{t^2}<0.
		\end{align*}
		Here, we used $\ln(1/\gamma(t))$ which follows from $\gamma(t)\le \gamma(t_0)\le 1$. Thus, the property \ref{ass:zeta_pic} holds. \smallskip
		
		We calculate the third derivative of $\zeta_2$. First, we assume $s\ne 1/2$. The term $\zeta_2^{(3)}(t)$ equals
		\begin{align*}
			&C\frac{ -2 \big(2s\gamma(t)^{2s-1}-1\big)t\gamma'(t)+2\big(\gamma(t)^{2s}-\gamma(t)\big)  +\gamma^{\prime\prime}(t)t^2\big( 2s\gamma(t)^{2s-1}-1 \big)}{(1-2s)t^3}\\
			&\quad- \frac{C 2s t^2 (\gamma'(t))^2\gamma(t)^{2s-2}}{t^3}=: \frac{\I{}+\II{}}{t^2}.
		\end{align*}
		Note that 
		\begin{equation}\label{eq:Z2_without_pUB_help}
			\frac{\gamma(t)^{2s-1}-1}{1-2s}\ge \frac{2s\gamma(t)^{2s-1}-1}{1-2s}
		\end{equation}
		holds for all $t>0$. Using \eqref{eq:Z2_without_pUB_help}, the concavity of $\gamma$, and the previous representation of $\zeta_2^\prime$, we can bound $\I{}$ from below by 
		\begin{align*}
			-2 \frac{t\gamma'(t)}{\gamma(t)}\zeta_2^{\prime}(t)+2\zeta_2^{\prime}(t)+\frac{\gamma^{\prime\prime}(t)t^2}{\gamma(t)}\zeta_2^{\prime}(t)\ge-5 \frac{t\gamma'(t)}{\gamma(t)}\zeta_2^{\prime}(t)+\zeta_2^{\prime}(t)\ge -4\zeta_2^{\prime}(t).
		\end{align*}
		Here, we used \eqref{eq:modulus_second_derivative} and the concavity of $\gamma$. Due to the previous representation of $\zeta_2^{\prime}$, the bound 
		\begin{equation*}
			\II{}\ge -2s \zeta_2^{\prime}(t)
		\end{equation*}
		follows from 
		\begin{align*}
			\gamma(t)^{2s}\le \frac{\gamma(t)^{2s}-\gamma(t)}{1-2s}.
		\end{align*}
		This is true due to the assumption $\gamma(t)\le \gamma(t_1)\le (2s)^{1/(1-2s)}$. Thus, the property \ref{ass:zeta_3rd} follows in the case $s\ne 1/2$. 
		
		Now, we assume $s=1/2$. In this case, the third derivative of $\zeta_2$ reads
		\begin{align*}
			\frac{C}{t^3}\Big( t+2(\gamma(t)-t\gamma'(t))\ln(1/\gamma(t))+ t\frac{\gamma(t)-t\gamma'(t)}{\gamma(t)}+t^2\gamma^{\prime\prime}(t)\ln(1/\gamma(t)) \Big).
		\end{align*}
		Since $\gamma$ is concave and by \eqref{eq:modulus_second_derivative}, this is bounded from below by 
		\begin{equation*}
			\frac{C}{t^3}\big( -\gamma(t)-3t\gamma'(t) \big)\ln(1/\gamma(t))\ge -\frac{4}{t^2}\frac{C\gamma(t)\ln(1/\gamma(t))}{t}= -4 \frac{\zeta_2^\prime(t)}{t^2}.  
		\end{equation*}
		This is the desired bound and, thus, the property \ref{ass:zeta_3rd} follows also in the case $s=1/2$. \smallskip
		
		Now, we check property \ref{ass:zeta_ndecr}. If $s\ne 1/2$, then the derivative of $t\zeta_2'(t)$ equals
		\begin{align*}
			\gamma'(t)\frac{2s\gamma(t)^{2s-1}-1}{1-2s}
		\end{align*}
		which is nonnegative if $\gamma(t)\le (2s)^{1/(1-2s)}$. If $s=1/2$, then 
		\begin{align*}
			\partial_t (t\gamma'(t))= \gamma'(t)\ln(1/\gamma(t))-\gamma'(t)
		\end{align*}
		which is nonnegative if $\gamma(t)\le 1/e$. Thus, \ref{ass:zeta_ndecr} is true.\smallskip
		
		For the moment, we assume $s\ne 1/2$. We calculate the derivative of $\zeta(t)/t^{\iota}$. It equals
		\begin{align*}
			\frac{1}{t^{\iota+1}}\Big( \gamma(t)\frac{\gamma(t)^{2s-1}-1}{1-2s}-\iota\int\limits_0^t \frac{\gamma(r)^{2s-1}-\gamma(r)}{1-2s}\frac{\gamma(r)}{r}\d r \Big).
		\end{align*}
		By \eqref{eq:Z2_without_pUB_help} and \eqref{eq:modulus_decreasing_iota} for $\gamma$, we conclude
		\begin{align*}
			\iota\frac{\gamma(r)^{2s-1}-\gamma(r)}{1-2s}\frac{\gamma(r)}{r}\ge \frac{2s\gamma(r)^{2s-1}-\gamma(r)}{1-2s}\gamma'(r)=  \partial_r\big(\frac{\gamma(r)^{2s-1}-\gamma(r)}{1-2s}\gamma(r)\big).
		\end{align*}
		This bound yields $\partial_t \zeta(t)/t^{\iota}\le 0$ and, thus, the assumption \ref{ass:zeta_iota} holds for the function $\gamma_2$ in the case $s\ne 1/2$. Now, if $s=1/2$, then using \eqref{eq:modulus_decreasing_iota}
		\begin{equation*}
			\iota \frac{\gamma(r)}{r}\ln(1/\gamma(r))\ge \gamma'(r)\ln(1/\gamma(r))\ge\gamma'(r)\big(\ln(1/\gamma(r))-1\big)= \partial_r \big(\gamma(r)\ln(1/\gamma(r))\big). 
		\end{equation*}
		This yields
		\begin{align*}
			t^{\iota+1}\,\partial_t \frac{\zeta(t)}{t^{\iota}}&= \gamma(t)\ln(1/\gamma(t))- \iota \int\limits_{0}^t \frac{\gamma(r)\ln(1/\gamma(r))}{r}\d r\ge 0
		\end{align*}
		and, thus, \ref{ass:zeta_iota} holds also for $s=1/2$. \smallskip
		
		\textit{Step 2.} Finally, we extend $\zeta$ to $[0,\infty)$ such that \ref{ass:zeta_pic}, \ref{ass:zeta_ndecr}, and \ref{ass:zeta_growth} remain true. For this we refer to step 2 in the proof of \autoref{lem:zeta_with_pUB}. 
	\end{proof}
		
	The next four propositions provide super and sub-solutions function which we will use as barriers in \autoref{sec:boundary_regularity_dini}. 
	
	\begin{proposition}\label{prop:supharmonic_barrier_pUB}
		Let $D\subset \R^d$ be a bounded uniform $\cdini$-domain. Suppose \eqref{ass:pUB} holds. For any $s_0\in (0,1)$ there exist constants $\eps_0>0$, $C_1,C_2>0$ such that for any $s\in[s_0,1)$ there exists a function $b_+\in H^s(\R^d)$ satisfying 
		\begin{align*}
			A_s b_+&\ge 1 \text{ in }\{ x\in D\mid \distw{x}{D^c}<\eps_0 \},\\
			C_1\distw{\cdot}{D^c}^s&\le b_+ \le C_2\distw{\cdot}{D^c}^s.
		\end{align*}
	\end{proposition}
	\begin{proof}
		Let $\omega$ be the function obtained from \autoref{lem:modulus} applied to the uniform modulus of continuity of the charts of $\partial D$. From \autoref{lem:zeta_with_pUB}, we deduce the existence of $\zeta$ satisfying \ref{ass:zeta_pic}-\ref{ass:zeta_growth} and \eqref{eq:zeta_omega_coupling_pUB} for the constant $C:= (c_1C_5+\tilde{C}+1)/(c_1C_4)$ where $\tilde{C}$ is taken from \autoref{prop:almost_harmonic}, $C_4, C_5$ from \autoref{prop:subsolution_d}, and $c_1:= 1/(2\zeta(\diam(D)))$. We define 
		\begin{equation*}
			b_+(x):= \fd(x)^s-c_1 \fd(x)^s\zeta(\fd(x)).
		\end{equation*}
		By the choice of $c_1$, we know $b_+>0$ in $D$ and $b_+(x)\le \fd(x)^s\le c_2^s \distw{x}{D^c}^s$ by \eqref{eq:reg_distance_properties} as well as $b_+(x)\ge 1/2 \fd(x)^s\ge c_2^{-s}/2 \, d_x^s$. Using these properties and \cite[Lemma B.2.5]{RoFe24}, we know $b_+\in H^s(\R^d)$. By \autoref{prop:almost_harmonic} and \autoref{prop:subsolution_d} and the choice of $\zeta$ and the constant $C$, we find 
		\begin{align*}
			A_s b_+(x)&\ge -(c_1C_5 +\tilde{C})\Big( \frac{\omega(d_x)}{d_x^s}+1\Big)+c_1 C_4\frac{\zeta^\prime(d_x)}{d_x^{s-1}}\ge 1.
		\end{align*}
		Here, we wrote $d_x:= \distw{x}{D^c}$ for short. 
	\end{proof}
	\begin{proposition}\label{prop:supharmonic_barrier_without_pUB}
		Let $s_0\in (0,1)$, $D\subset \R^d$ be a bounded uniform $2s_0$-$\cdini$-domain. There exist constants $\eps_0>0$, $C_1,C_2>0$ such that for any $s\in[s_0,1)$ there exists a function $b_+\in H^s(\R^d)$ satisfying 
		\begin{align*}
			A_s b_+&\ge 1 \text{ in }\{ x\in D\mid \distw{x}{D^c}<\eps_0 \},\\
			C_1\distw{\cdot}{D^c}^s&\le b_+ \le C_2\distw{\cdot}{D^c}^s.
		\end{align*}
	\end{proposition}
	\begin{proof}
		The proof uses the same arguments as the previous one for \autoref{prop:supharmonic_barrier_pUB} but instead of \autoref{lem:zeta_with_pUB} we use \autoref{lem:zeta_without_pUB}. Here, note that \autoref{lem:special_2s_0_dini_bigger_s}.
	\end{proof}

	The following two propositions provide appropriate sub-solutions which will be used as barriers to prove a Hopf lemma. 
	
	\begin{proposition}\label{prop:subharmonic_barrier_pUB}
		Let $D\subset \R^d$ be a bounded uniform $\cdini$-domain. Suppose \eqref{ass:pUB} holds. For any $s_0\in (0,1)$ there exist constants $\eps_0>0$, $C_1,C_2>0$ such that for any $s\in[s_0,1]$ there exists a function $b_-\in H^s(\R^d)$ satisfying 
		\begin{align*}
			A_s b_-&\le -1 \text{ in }\{ x\in D\mid \distw{x}{D^c}<\eps_0 \},\\
			C_1\distw{\cdot}{D^c}^s&\le b_+ \le C_2\distw{\cdot}{D^c}^s.
		\end{align*}
	\end{proposition}
	\begin{proof}
		 The proof follows the same lines as the proof of \autoref{prop:supharmonic_barrier_pUB} but instead of $b_+$ we consider
		\begin{equation*}
			b_-:=\fd^s+\fd^s\zeta(\fd).
		\end{equation*} 
	\end{proof}
	\begin{proposition}\label{prop:subharmonic_barrier_without_pUB}
		Let $s_0\in (0,1)$, $D\subset \R^d$ be a bounded uniform $2s_0$-$\cdini$-domain. There exist constants $\eps_0>0$, $C_1,C_2>0$ such that for any $s\in[s_0,1)$ there exists a function $b_-\in H^s(\R^d)$ satisfying 
		\begin{align*}
			A_s b_-&\le -1 \text{ in }\{ x\in D\mid \distw{x}{D^c}<\eps_0 \},\\
			C_1\distw{\cdot}{D^c}^s&\le b_+ \le C_2\distw{\cdot}{D^c}^s.
		\end{align*}
	\end{proposition}
	\begin{proof}
		The proof follows the same lines as the proof of \autoref{prop:supharmonic_barrier_without_pUB} but instead of $b_+$ we consider
		\begin{equation*}
			b_-:=\fd^s+\fd^s\zeta(\fd).
		\end{equation*} 
	\end{proof}

	\section{Boundary estimates}\label{sec:boundary_regularity_dini}
	In this section, we prove both the $C^s$-boundary regularity, see \autoref{th:boundary_regularity} and \autoref{th:boundary_regularity_without_pUB}, and the Hopf boundary lemma, see \autoref{sec:hopf}. We combine the barriers build in \autoref{sec:higher_dimensional} with the interior regularity from \autoref{sec:interior}.
	
	\subsection{Boundary regularity}
	The first step is an upper bound on $\abs{u}$ in terms of the distance to a boundary point. This is done in the next proposition. The main ingredient are the barriers constructed in \autoref{prop:supharmonic_barrier_pUB} respectively \autoref{prop:supharmonic_barrier_without_pUB} combined with a maximum principle. 
	\begin{proposition}\label{prop:distance_upper_bound_with_pUB}
		Let $\Omega \subset \R^d$ be an open set such that $\Omega\cap B_2$ satisfies the exterior $\cdini$-property at a boundary point $z\in B_{1/2}\cap \partial \Omega$. Let $s_0\in (0,1)$. Assume \eqref{ass:pUB} or $s_0>1/2$. There exists a positive constant $C$ such that for any $f\in L^\infty(\Omega\cap B_1)$, any $s\in [s_0,1)$, and any weak solution $u$ to 
		\begin{equation}\label{eq:main_equation_stable_localized}
			\begin{split}
				A_\mu^s u &= f \text{ in }\Omega\cap B_{1},\\
				u &= 0 \text{ on }\Omega^c \cap B_{2},
			\end{split}
		\end{equation}
		 the solution $u$ satisfies the bound
		\begin{equation}\label{eq:distance_upper_bound}
			\abs{u(x)}\le C\, \big( \norm{f}_{L^\infty}+\sup_{y\in B_1}\tail_{\nu_s}(u;y)+\norm{u}_{L^\infty(B_2)}\big) \abs{x-z}^s \quad\text{ for  } x\in B_1,
		\end{equation}
		where the tail-term is defined by \eqref{eq:tail}.
	\end{proposition}
	\begin{remark}
		The \autoref{prop:distance_upper_bound_with_pUB} is known to be false for Lipschitz domains. This can be seen using the theory on $s$-harmonic functions on infinite cones, see \cite{BaBo04} and \cite[Lemma 3.3]{Mic06}.
	\end{remark}
	\begin{proof}
		If the right-hand side of \eqref{eq:distance_upper_bound} is not finite, then the bound is trivial. We assume the contrary. We fix a smooth cutoff function $\phi$ with $\phi=1$ in $B_{3/2}$ and $\phi=0$ on $B_{2}^c$. Further, we define $\overline{u}:= u\phi$. Then, $\overline{u}$ satisfies
		\begin{equation}\label{eq:truncated_equation}
			A_\mu^s\overline{u}(x)= f(x)+\int\limits_{\{-x\}+B_{3/2}^c} (1-\phi(x+h))u(x+h)\nu_s(\d h) \text{ in }\Omega \cap B_{1}\text{ in the weak sense.}
		\end{equation}
		Let $D_z$ be the $\cdini$-domain adapted to $z$ and $\Omega$ by the exterior $\cdini$-property. Further, let $b_+$ be the function and $\eps_0, C_1, C_2$ be the constants from \autoref{prop:supharmonic_barrier_pUB} if \eqref{ass:pUB} holds and the function from \autoref{prop:supharmonic_barrier_without_pUB} if \eqref{ass:pUB} is not true but $s_0>1/2$ holds. We define 
		\begin{align*}
			\psi:= \big( \norm{f}_{L^\infty}+ \sup_{y\in B_1}\tail_{\nu_s}(u;y)+\frac{\norm{\overline{u}}_{L^\infty}}{C_1\eps_0^s} \big)b_+.
		\end{align*}
		For $x\in \R^d$ such that $\distw{x}{D_z^c}\ge \eps_0$ or $x\in \Omega^c\cap B_2$, we know
		\begin{align*}
			\psi(x)\ge \frac{\distw{x}{D_z^c}^s}{\eps_0^s}\norm{\overline{u}}_{L^\infty}\ge \overline{u}(x)
		\end{align*}
		by \autoref{prop:supharmonic_barrier_pUB} or \autoref{prop:supharmonic_barrier_without_pUB} respectively. Using the same propositions and \eqref{eq:truncated_equation}, we find in $B_1\cap \Omega$
		\begin{equation*}
			A_\mu^s[\overline{u}-\psi](x)\le 0 
		\end{equation*}
		in $B_1\cap \Omega$ in the weak sense. Thus, the maximum principle for weak solution, see e.g.\ \cite[Theorem 5.1]{Rut18}, yields for any $x\in B_1$
		\begin{align*}
			u(x)&=\overline{u}(x)\le \psi(x)\le C_2 \big( \norm{f}_{L^\infty}+ \sup_{y\in B_1}\tail_{\nu_s}(u;y)+\frac{\norm{\overline{u}}_{L^\infty}}{C_1\eps_0^s} \big)\distw{x}{D_z^c}^s\\
			&\le C\big( \norm{f}_{L^\infty}+ \sup_{y\in B_1}\tail_{\nu_s}(u;y)+\frac{\norm{\overline{u}}_{L^\infty}}{C_1\eps_0^s} \big)\abs{x-z}^s.
		\end{align*}
		By symmetry, we may repeat this procedure for $-u$ instead of $u$.
	\end{proof}

	\begin{proposition}\label{prop:distance_upper_bound_without_pUB}
		Let $s_0\in (0,1)$. We fix an open set $\Omega \subset \R^d$ such that $\Omega\cap B_2$ satisfies the exterior $2s_0$-$\cdini$-property with the same modulus of continuity at a boundary point $z\in B_{1/2}\cap \partial \Omega$. There exists a positive constant $C$ such that for any $f\in L^\infty(\Omega\cap B_1)$, any $s\in [s_0,1)$, and any weak solution $u$ to \eqref{eq:main_equation_stable_localized} the estimate \eqref{eq:distance_upper_bound} holds.
	\end{proposition}
	\begin{proof}
		After \autoref{lem:special_2s_0_dini_bigger_s}, the proof uses the same ideas as the one of \autoref{prop:distance_upper_bound_with_pUB} but with \autoref{prop:supharmonic_barrier_without_pUB} instead of \autoref{prop:supharmonic_barrier_pUB}. 
	\end{proof}

	The next lemma yields an global $L^\infty$-bound. This result is well known, see e.g.\ \cite[Lemma 2.3.9, Lemma 2.3.10]{RoFe24}, but note it is essential for us that the constant in \eqref{eq:Linfty_bound} depends on $s$ only through a lower bound $s_0\le s$. For this reason, we provide the proof here.
	
	\begin{lemma}[$L^\infty$-bound]\label{lem:Linfty_bound}
		Let $s_0\in (0,1)$, $\Omega\subset \R^d$ be bounded, and $0<\lambda\le \Lambda<\infty$. There exists a positive constant $C=C(s_0,\lambda,\Lambda, d)$ such that for any $f\in L^\infty(\Omega)$, $s_0\le s<1$, and a weak solution $u$ to \eqref{eq:main_equation_stable} the function $u$ satisfies the bound
		\begin{equation}\label{eq:Linfty_bound}
			\norm{u}_{L^\infty(\R^d)}\le C(\diam(\Omega))^{2s}\norm{f}_{L^\infty(\Omega)}.
		\end{equation}
	\end{lemma}
	\begin{proof}
		We prove the result in two step. Without loss of generality we assume $0\in \Omega$.\smallskip
		
		\textit{Step 1.} We construct a super-solution. We define $\psi(x):=(1-x^2)_+$ similar to \cite[Lemma 2.3.10]{RoFe24}. This function satisfies 
		\begin{align*}
			A_{\mu}^s\psi(0)&= \mu(S^{d-1})2(1-s)\Big(\int\limits_{0}^1 r^{1-2s}\d r + \int\limits_{1}^{\infty}r^{-1-2s}\d r\Big)= \frac{\mu(S^{d-1})}{s}\ge \lambda.
		\end{align*} 
		Since $A_{\mu}^s \psi$ is continuous, which is due to the regularity and boundedness of $\psi$, we find a positive radius $r_1=r_1(s_0,\lambda,\Lambda,d)$ such that for any $x\in \R^d$ with $\abs{x}\le r_1$ 
		\begin{equation}\label{eq:Linfty_barrier_1}
			A_\mu^s \psi(x)\ge  \frac{\lambda}{2}.
		\end{equation}
		Now, we define a scaled version of $\psi$, $\psi_R(x):= c_1\psi_1(x/R)$ for $R>0$, $c_1>0$. This function satisfies
		\begin{align*}
			A_{\mu}^s \psi_R(x)= c_1R^{-2s}A_{\mu}^s\psi_1(x/R)\ge c_1\frac{R^{-2s}\lambda}{2}
		\end{align*}
		for any $x\in B_{Rr_1}(0)$ by \eqref{eq:Linfty_barrier_1}.\smallskip
		
		\textit{Step 2.} We compare the super-solution from step 1 with $u$. Fix $R:=\diam(\Omega)/r_1$ and $c_1:= 2 R^{2s}\,\norm{f}_{L^\infty(\Omega)}/\lambda$. Then, $A_{\mu}^s[\psi_R-u]\ge 0$ in $\Omega$ and, since $u=0$ on $\Omega^c$, $\psi_R-u\ge 0$ on $\Omega^c$. The weak maximum principle for weak solutions, see e.g.\ \cite[Theorem 5.1]{Rut18}, yields $u\le \psi_R$. Repeating this argument with $-u$ yields
		\begin{align*}
			\abs{u(x)}\le \psi_R(x)\le c_1= \frac{2}{\lambda r_1^{2s}}  \diam(\Omega)^{2s}\,\norm{f}_{L^\infty(\Omega)}.
		\end{align*}
	\end{proof}

	\begin{corollary}[Global boundary estimate]\label{cor:global_boundary}
		Suppose either all assumptions from \autoref{th:boundary_regularity} or \autoref{th:boundary_regularity_without_pUB} hold. There exists a positive constant $C= C(s_0,\lambda,\Lambda, \Omega)$ such that any weak solution $u$ to \eqref{eq:main_equation_stable} with $f\in L^\infty(\Omega)$ satisfies
		\begin{equation*}
			\abs{u}\le C \norm{f}_{L^\infty(\Omega)} d_\Omega^s
		\end{equation*}
	\end{corollary}
	\begin{proof}
		Firstly, note that the tail term in the bound \eqref{eq:distance_upper_bound} can be estimated by $\norm{u}_{L^\infty}$. If \eqref{ass:pUB} holds or $s_0>1/2$, then we use \autoref{prop:distance_upper_bound_with_pUB}. Else, the \autoref{prop:distance_upper_bound_without_pUB} applied to all $z\in \partial \Omega$ yields
		\begin{align*}
			\abs{u}\le c_1 \big( \norm{f}_{L^\infty} + \norm{u}_{L^\infty(\R^d)} \big) d_\Omega^s.
		\end{align*}
		Now, the result follows from \autoref{lem:Linfty_bound}. 
	\end{proof}
	
	Using \autoref{th:interior_s_regularity}, \autoref{lem:Linfty_bound}, and \autoref{cor:global_boundary}, the proofs of \autoref{th:boundary_regularity} and \autoref{th:boundary_regularity_without_pUB} follows with arguments very close to \cite[Proposition 1.1]{RoSe14}.
	\begin{proof}[Proof of \autoref{th:boundary_regularity} and \autoref{th:boundary_regularity_without_pUB}.]
		We prove the results in two steps. \smallskip
		
		\textit{Step 1.} First we prove the bound
		\begin{equation}\label{eq:proof_main_theorem_claim_1}
			\frac{\abs{u(x)-u(y)}}{\abs{x-y}^{s}}\le c_1 \norm{f}_{L^\infty(\Omega)}.
		\end{equation}
		for any $x,y\in \Omega$ such that $\abs{x-y}\le d_{\Omega}(x)/4$. \smallskip
		
		We consider the function $v(h):= u(x+d_\Omega(x)/2\,h)$. This function satisfies 
		\begin{equation*}
			A_\mu^sv(h)= \frac{d_{\Omega}(x)^{2s}}{2^{2s}} f(x+d_\Omega(x)/2\,h)=:\tilde{f}(h)
		\end{equation*}
		for $y\in B_{3/2}$ in the weak and, thus, also in the distributional sense. The function $v$ is bounded due to \autoref{lem:Linfty_bound}. Therefore, the \autoref{th:interior_s_regularity} applied to $v$ yields
		\begin{equation}\label{eq:proof_main_theorem_step1}
			\begin{split}
				\frac{\abs{u(x)-u(y)}}{\abs{x-y}^{s}}&\le \frac{2^s}{d_\Omega(x)^s} [v]_{C^s(B_{1/2})}\\
				&\le c_1 \frac{2^s}{d_\Omega(x)^s}\Big(\frac{d_{\Omega}(x)^{2s}}{2^{2s}}\norm{\tilde{f}}_{L^\infty(B_1)}+\norm{v}_{L^\infty(B_{2})}+ \sup_{h\in B_1}\tail_{\nu_s}(v;h)\Big)
			\end{split}
		\end{equation}
		for any $x,y\in \Omega$ such that $\abs{x-y}\le d_{\Omega}(x)/4$. By \autoref{cor:global_boundary}, we know
		\begin{align*}
			\norm{v}_{L^\infty(B_2)}=\norm{v}_{L^\infty(B_{d_\Omega(x)}(x))}\le c_2 \norm{\tilde{f}}_{L^\infty(\Omega)} 2^s d_\Omega(x)^s.
		\end{align*}
		Note that $d_\Omega(x+d_\Omega(x)/2\,h)\le d_{\Omega}(x)+\abs{x+d_\Omega(x)/2\, h - x}\le d_{\Omega}(x)(1+\abs{h})$ for any $h\in \R^d$. Thus, again due to \autoref{cor:global_boundary} applied to $u$ we find 
		\begin{align*}
			\abs{v(h)}\le c_2 \norm{f}_{L^\infty(\Omega)} d_{\Omega}(x+d_{\Omega}(x)/2\, h)^s\le  c_2 d_\Omega(x)^{s}(1+\abs{h})^s\norm{f}_{L^\infty(\Omega)}.
		\end{align*}
		This allows us to estimate the tail term as follows:
		\begin{align*}
			\sup\limits_{h\in B_1}\tail_{\nu_s}(v;h)&\le c_2 d_{\Omega}(x)^s \sup\limits_{h\in B_1}\int\limits_{B_{1/2}(0)^c} (1+\abs{h+r\theta})^s \nu_s(\d \theta)\\
			&= 2c_2 d_{\Omega}(x)^s (1-s)\mu(S^{d-1})\int\limits_{1/2}^\infty \frac{(2+\abs{r})^s}{\abs{r}^{1+2s}}\d r \\
			&\le 2\, 5^s d_{\Omega}(x)^s (1-s)\Lambda\frac{(1/2)^{-s}}{s}
		\end{align*}
		These observations together with \eqref{eq:proof_main_theorem_step1} yield the claim. \medskip
		
		\textit{Step 2.} Now, we prove that the inequality \eqref{eq:proof_main_theorem_claim_1} with a slightly bigger constant holds for all $x,y\in \R^d$. \smallskip
				
		If $x\notin \Omega$ or $y\notin \Omega$, then this follows directly from \autoref{cor:global_boundary}. It remains to prove the bound for $x,y\in \Omega$ such that $\abs{x-y}>\max\{d_{\Omega}(x), d_\Omega(y)\}/4$. In this, case the estimate again follows from \autoref{cor:global_boundary}. Using the triangle inequality and \autoref{cor:global_boundary}, we write
		\begin{equation*}
			\abs{u(x)-u(y)}\le \abs{u(x)}+\abs{u(y)}\le c_3 \norm{f}_{L^\infty}\big( d_\Omega(x)^s+d_\Omega(y)^s \big)\le 2 c_3 4^s \norm{f}_{L^\infty} \, \abs{x-y}^{s}.
		\end{equation*}
	\end{proof}
		
	\subsection{Hopf boundary lemma}\label{sec:hopf}
	
	\begin{proposition}[localized Hopf-type boundary lemma I]\label{prop:distance_lower_bound_with_pUB}
		Let $\Omega \subset \R^d$ be an open set such that $\Omega$ satisfies the interior $\cdini$-property at a boundary point $z\in B_{1/2}\cap \partial \Omega$. Let $s_0\in (0,1)$. Assume \eqref{ass:pUB} or $s_0>1/2$. There exist positive constants $C, \eps_0$ such that for any $s\in [s_0,1)$ and any distributional super-solution $u\in C(B_1\cap \Omega)$ to \eqref{eq:frac_lap_sup_solution_localized} we know that 
		\begin{equation}\label{eq:distance_lower_bound_local}
			u(x)\ge \abs{x-z}^s\,C\,\inf\{ u(y)\mid y\in B_{1/2}\cap \Omega, \distw{y}{\partial \Omega}\ge \eps_0 \} 
		\end{equation}
		for any $x\in B_{1/2}\cap \Omega$ inside of a non-tangential cone with apex at $z$. 
	\end{proposition}
	\begin{proof}
		If $u=0$, then there is nothing to show. So we assume the contrary. By the strong maximum principle for continuous solutions, see \cite{RoFe24}, we know $u>0$ in $\Omega\cap B_1$. Due to another application of the maximum principle, we we assume without loss of generality that $u=0$ on $(\Omega\cap B_1)^c$. Let $D_z$ be the $\cdini$-domain adapted to $z$ and $\Omega$ by the interior $\cdini$-property. We assume that $D_z$ fits into $B_{1/2}\cap \Omega$, i.e.\ after translation and scaling. Further, let $b_-$ be the function and $\eps_0, C_1, C_2$ be the constants from \autoref{prop:subharmonic_barrier_pUB} if \eqref{ass:pUB} holds and the function from \autoref{prop:subharmonic_barrier_without_pUB} if $s_0>1/2$ but not \eqref{ass:pUB}. Recall \autoref{rem:special_dini_s_ge_12}. \smallskip
		
		Since $u$ is continuous and positive in $\Omega\cap B_1$, the constant 
		\begin{equation*}
			c_u:= \inf\{ u(y)\mid y\in B_{1/2}\cap D_z, \distw{y}{\partial D_z}\ge \eps_0 \}
		\end{equation*}
		is positive. We define $\psi:= \frac{c_u}{C_2} b_-$. By the choice of $c_u$ and the properties of $b_-$ we deduce $\psi\le u$ in $\{ u(y)\mid y\in D_z, \distw{y}{\partial (D_z\cap B_1)}\ge \eps_0 \}$. This may be seen as follows. If $y\notin D_z$, then $b_-=0$ and the estimate is trivial by the nonnegativity of $u$. If $y\in D_z$ but $\distw{y}{D_z^c}\ge \eps_0$, then 
		\begin{equation*}
			u(y)\ge c_u\ge\frac{c_u}{C_2} C_2(1/2)^s\ge \frac{c_u}{C_2} C_2(\diam(D_z))^s\ge \frac{c_u}{C_2} C_2\distw{y}{D_z^c}^s\ge \frac{c_u}{C_2} b_-(y)= \psi(y).
		\end{equation*}
		Furthermore, using the properties of $b_-$ from \autoref{prop:subharmonic_barrier_pUB} respectively \autoref{prop:subharmonic_barrier_without_pUB} we find
		\begin{align*}
			As[u-\psi](x)\ge 1 \text{ in } \{ y\in D_z\mid \distw{y}{D_z^c}<\eps_0 \} \text{ in the distributional sense.}
		\end{align*}
		The maximum principle, see \cite{Sil07}, \cite[Lemma 2.3.5]{RoFe24}, yields $u\ge \psi$ which implies the statement. 
	\end{proof}
	\begin{proposition}[localized Hopf boundary lemma II]\label{prop:distance_lower_bound_without_pUB}
		Let $s_0\in (0,1)$. We fix an open set $\Omega \subset \R^d$ such that $\Omega$ satisfies the interior $2s_0$-$\cdini$-property at a boundary point $z\in B_{1/2}\cap \partial \Omega$. Then there exist positive constants $C, \eps_0$ such that for any $s\in [s_0,1)$ and any distributional super-solution $u\in C(B_1\cap \Omega)$ to \eqref{eq:frac_lap_sup_solution_localized} satisfies \eqref{eq:distance_lower_bound_local}. 
	\end{proposition}
	\begin{proof}
		After acknowledging \autoref{lem:special_2s_0_dini_bigger_s}, the proof uses the same ideas as the one of \autoref{prop:distance_upper_bound_with_pUB} but always using \autoref{prop:supharmonic_barrier_without_pUB} instead of \autoref{prop:supharmonic_barrier_pUB}. 
	\end{proof}
	\begin{proof}[Proof of \autoref{prop:hopf_with_pUB}]
		If $u$ is zero, then there is nothing to show. If $u$ is not zero, then $u$ is positive in the interior of $\Omega\cap B_1$ by the strong maximum principle. Thus, the result follows from \autoref{prop:distance_lower_bound_with_pUB} or \autoref{prop:distance_lower_bound_without_pUB}. 
	\end{proof}
	\begin{corollary}[global Hopf boundary lemma I]\label{prop:global_hopf_pUB}
		Let $\Omega \subset \R^d$ be an bounded open set such that $\Omega$ satisfies the interior $\cdini$-property at every boundary point $\partial \Omega$ with the same modulus of continuity. Let $s_0\in (0,1)$. Assume \eqref{ass:pUB} or $s_0>1/2$. There exist positive constants $C, \eps_0$ such that for any $s\in [s_0,1)$ and any distributional super-solution $u\in C(\Omega)$ to 
		\begin{equation}\label{eq:frac_lap_sup_solution_global}
			\begin{split}
				A_s u &\ge 0 \text{ in }\Omega,\\
				u &\ge 0 \text{ on }\Omega^c,
			\end{split}
		\end{equation}
		we know that
		\begin{equation}\label{eq:distance_lower_bound_global}
			u(x)\ge d_\Omega(x)^s\,C\,\inf\{ u(y)\mid y\in \Omega, \distw{y}{\partial \Omega}\ge \eps_0 \} 
		\end{equation}
		for any $x\in \Omega$. 
	\end{corollary}
	\begin{proof}
		Using the maximum principle for continuous distributional solutions \cite{Sil07}, \cite[Lemma 2.3.5]{RoFe24}, we find $u\ge 0$ in $\R^d$. Now, an application of \autoref{prop:distance_lower_bound_with_pUB} at every boundary point $z\in \partial \Omega$ yields the desired result.
	\end{proof}
	\begin{corollary}[global Hopf boundary lemma II]\label{prop:global_hopf_without_pUB}
		Let $s_0\in (0,1)$. We fix an open set $\Omega \subset \R^d$ such that $\Omega$ satisfies the interior $2s_0$-$\cdini$-property uniformly at every boundary point in $\partial \Omega$. Then there exist positive constants $C, \eps_0$ such that for any $s\in [s_0,1)$ and any weak super-solution $u\in C(\Omega)$ to \eqref{eq:frac_lap_sup_solution_global} satisfies \eqref{eq:distance_lower_bound_global}.
	\end{corollary}
	

\end{document}